\documentclass[10pt]{amsart}

\usepackage{amssymb, amsmath}
\usepackage{mathrsfs}
\usepackage{amscd}
\usepackage{enumerate}

\usepackage[colorlinks,linkcolor={blue},citecolor={blue},urlcolor={red}]{hyperref}

\usepackage{newsymbol}
\let\emptyset \undefined
\let\ge       \undefined
\let\le       \undefined
\newsymbol\le          1336  \let\le\le
\newsymbol\ge          133E  \let\geq\ge
\newsymbol\emptyset    203F
\newsymbol\notle       230A
\newsymbol\notge       230B

\theoremstyle{plain}
\newtheorem{theorem}{Theorem}[section]
\theoremstyle{remark}
\newtheorem{remark}[theorem]{Remark}

\newtheorem{definition}[theorem]{Definition}
\theoremstyle{plain}
\newtheorem{corollary}[theorem]{Corollary}
\newtheorem{lemma}[theorem]{Lemma}
\newtheorem{proposition}[theorem]{Proposition}

\newtheorem{assumption}[theorem]{Assumption}
\numberwithin{equation}{section}

\def\Z{{\mathbb Z}}

\def\R{{\mathbb R}}
\def\C{{\mathbb C}}

\newcommand{\E}{{\mathbb E}}
\renewcommand{\P}{{\mathbb P}}
\newcommand{\F}{{\mathscr F}}

\renewcommand{\a}{\alpha}
\renewcommand{\b}{\beta}
\newcommand{\g}{\gamma}

\newcommand{\e}{\varepsilon}

\renewcommand{\O}{\Omega}

\newcommand{\abs}[1]{\left|#1\right|}
\newcommand{\ud}[0]{\,\mathrm{d}}
\newcommand{\intav}[0]{-\!\!\!\!\!\!\int}
\newcommand{\dydt}[0]{\tfrac{\ud y\ud t}{t^{n+1}}}

\newcommand{\ran}{\mathsf{R}}
\newcommand{\dom}{\mathsf{D}}

\newcommand{\calL}{{\mathscr L}}
\newcommand{\n}{\Vert}

\newcommand{\s}{^*}
\newcommand{\lb}{\langle}
\newcommand{\rb}{\rangle}

\newcommand{\limn}{\lim_{n\to\infty}}

\newcommand{\sumn}{\sum_{n\ge 1}}

\newcommand{\sumj}{\sum_{j\ge 1}}

\newcommand{\beq}{\begin{equation}}
\newcommand{\eeq}{\end{equation}}
\newcommand{\bal}{\begin{aligned}}
\newcommand{\eal}{\end{aligned}}
\newcommand{\ben}{\begin{enumerate}}
\newcommand{\een}{\end{enumerate}}
\newcommand{\bit}{\begin{itemize}}
\newcommand{\eit}{\end{itemize}}

\begin{document}

\author{Tuomas Hyt\"onen}
\address{Department of Mathematics and Statistics\\
University of Helsinki\\
Gustaf H\"all\-str\"omin katu 2b\\
FI-00014 Helsinki\\ 
Finland}
\email{tuomas.hytonen@helsinki.fi}

\author{Jan van Neerven}
\address{
Delft Institute of Applied Mathematics\\
Delft University of Technology 
\\P.O. Box 5031\\
2600 GA Delft\\
The Netherlands}
\email{J.M.A.M.vanNeerven@tudelft.nl}

\author{Pierre Portal}
\address{
Mathematical Sciences Institute\\
Building 27\\
Australian National University\\
ACT 0200\\
Australia} 
\email{Pierre.Portal@maths.anu.edu.au}

\title[Conical square functions in UMD Banach spaces]
{Conical square function estimates in UMD Banach spaces and applications 
to $H^{\infty}$-functional calculi}

\begin{abstract}
We study conical square function estimates for Banach-valued functions, 
and introduce a vector-valued analogue of the Coifman--Meyer--Stein tent 
spaces.
Following recent work of Auscher--M\(^{\mathrm{c}}\)Intosh--Russ, the tent 
spaces in turn are used to construct a scale of vector-valued Hardy spaces 
associated with a given bisectorial operator \(A\) with certain off-diagonal 
bounds, such that \(A\) always has a bounded \(H^{\infty}\)-functional 
calculus on these spaces.
This provides a new way of proving functional calculus of \(A\) 
on the Bochner spaces \(L^p(\R^n;X)\) by checking appropriate conical 
square function estimates, and also  a conical analogue of Bourgain's 
extension of the Littlewood-Paley theory to the UMD-valued context.
Even when \(X=\C\), our approach gives refined \(p\)-dependent versions 
of known results.
\end{abstract}


\date\today

\maketitle

\section{Introduction}
Since the development of the Littlewood-Paley theory, square function estimates of the form
\[ \Big\| \Big(\int_{0}
^{\infty}\big|t\sqrt{\Delta}e^{-t\sqrt{\Delta}}f\big|^{2}\frac{\ud t}{t}
\Big)^{\frac{1}{2}} \Big\|_{L^{p}(\R^n)} \eqsim \|f\|_{L^{p}(\R^n)},\]
have been widely used in harmonic analysis.
When dealing with functions which takes values in a UMD Banach space \(X\), such 
estimates have to be given an appropriate meaning. 
This is done through a linearisation of the square function
using randomisation, which gives (see \cite{lps-umd})
\[ \Big\|\int  _{0} ^{\infty} t\sqrt{\Delta}e^{-t\sqrt{\Delta}}f 
\frac{\ud Wt}{\sqrt{t}} \Big\|_{L^{2}(\Omega;L^p(\R^n;X))} \eqsim \|f\|_{L^{p}(\R^{n};X)},\]
where the integral is a  
Banach space-valued stochastic integral with respect to a standard 
Brownian motion $W$ on a probability space $(\Omega,\P)$ (see \cite{NW}), 
or,  in a simpler discrete form,
\begin{equation}
\label{bourgain-LP}  \Big\|\sum  _{k \in \Z} \varepsilon_{k}
2^{k}\sqrt{\Delta}e^{-2^{k}\sqrt{\Delta}}f \Big\|_{L^{2}(\Omega;L^p(\R^n;X))} 
\eqsim \|f\|_{L^{p}(\R^{n};X)}
\end{equation}
where \((\varepsilon_{k})\) is a sequence of independent Rademacher variables
on $(\Omega,\P)$. 
The latter was proven by Bourgain in \cite{Bou}, thereby starting the development of 
harmonic analysis for UMD-valued functions.
In recent years, research in this field has accelerated as 
it appeared that its tools, and in particular square function estimates, 
are of fundamental importance in the study of the \(H^{\infty}\)-functional 
calculus (see \cite{KuWe}) and in stochastic analysis in UMD Banach spaces 
(see \cite{NVW}).

To some extent, even the scalar-valued theory (i.e. $X=\C$) has benefited 
from this probabilistic point of view (see for instance \cite{HMP,leM}).
However this fruitful linearisation has, so far, been limited to the above 
``vertical" square functions estimates, leaving aside the ``conical" estimates 
of the form
\begin{equation} \label{conical-intro}
  \Big(\int  _{\R^{n}} \Big(\iint _{|y-x|<t}
  \big|t\sqrt{\Delta}e^{-t\sqrt{\Delta}}f(y)\big|^{2}\frac{\ud y\ud t}{t^{n+1}}\Big)^{\frac{p}{2}}\ud x\Big)^{\frac{1}{p}}
  \eqsim \|f\|_{L^{p}(\R^n)}, \quad 1<p\leq 2.
\end{equation}
In the meantime, such estimates have attracted much attention as it was
realised that they could be used to extend the real variable theory of Hardy
spaces in a way which is suitable to treat operators beyond the
Calder\'on-Zygmund class (see \cite{AMR, duong-yan,hoffman-mayboroda}). Indeed, elliptic operators of the form \(-{\rm div}
B\nabla\), where \(B\) is a matrix with $L^{\infty}$ entries, are not, in
general, sectorial on $L^{p}$ for all $1<p<\infty$. Their study thus requires
the $L^{p}$-spaces to be replaced by appropriate Hardy spaces, on which they have good functional calculus properties (in the same way as $L^{1}$ has to be replaced by $H^{1}$ when dealing with the Laplacian). To define such spaces, conical square functions have to be used, since the use of vertical ones would 
impose severe restrictions on the class of operators under consideration 
(namely, $L^{p}$ \hbox{($R$-)}sectoriality).

The present paper gives extensions of (\ref{conical-intro}) to the UMD-valued context.
This starts with the construction of appropriate tent spaces, which is carried out in Section \ref{sect:tent} by reinterpreting and extending \cite{HTV} using the methods of stochastic analysis in Banach spaces from \cite{KW,NVW,NW}. 
Relevant notions and results from this theory are recalled in Section \ref{sect:prelim}, while the crucial technical estimate is proven in Section \ref{sect:mainest}. 
Following ideas developed in \cite{AMR}, we then prove appropriate estimates for operators acting on these tent spaces in Section \ref{sect:od}. After collecting some basic results on bisectorial operator in Section \ref{sect:calculus}, this allows us in Section \ref{sect:hardy} to define Hardy spaces 
associated with bisectorial operators of the form $A \otimes I_{X}$, where \(A\) acts on $L^{2}(\R^n,H)$ ($H$ being a Hilbert space and $X$ a UMD Banach space) and satisfies suitable off-diagonal estimates.
We prove that $A \otimes I_{X}$ always has an $H^\infty$-functional calculus on these
Hardy spaces. 
Finally, in Section \ref{sect:diffop}, we specialise to differential operators \(A\), and, in particular, give a conical analogue to Bourgain's square function estimate (\ref{bourgain-LP}).

Specialising to the case \(X=\C\), our approach allows to define Hardy spaces 
(associated with operators) using a class of functions which is wider than in \cite{AMR}. This is due to the fact that our estimates (see Proposition \ref{prop:Schur}) are directly obtained for a given value of \(p\) (and actually depend on the type and cotype of \(L^{p}\)), instead of using interpolation.\\

To conclude this introduction, let us now point out the possible uses of our results.
First, one can deduce the boundedness of the functional calculus of an operator $A\otimes I_X$ from conical square function estimates. For instance, with Theorem \ref{thm:hisl}, we recover the well-known fact that,
if $X$ is UMD and $1<p<\infty$,  $\Delta\otimes I_X$ admits an
$H^\infty$-calculus on $L^p(X)$. Note that this
characterises the UMD spaces among all Banach spaces and thus indicates that it cannot be 
expexted that the results presented here extend beyond the UMD setting.

Another application is to deduce conical square function estimates for functions with limited decay from such estimates for functions with good decay properties.
In particular, Theorem \ref{thm:hisl} together with Theorem \ref{thm:hardy} give the following estimates: 
We use the notations
\begin{equation*}\begin{split}
  S_\theta^+ &= \{z\in \C\setminus\{0\}: \ |\arg(z)|<\theta\}, \\
  \Psi_{\alpha}^\beta(S_{\theta}^+)  &= 
  \big\{f \in H^{\infty}(S_{\theta}^+): \ \exists C \ \ |f(z)| \leq C\min(|z|^{\alpha},|z|^{-\beta})\ \forall\ z\in S_\theta^+\big\}.
\end{split}\end{equation*}
Let \(\theta,\varepsilon>0\), and assume that either
\begin{equation*}\begin{split}
  &\psi \in \Psi^{n/2+\varepsilon}_{1}(S_{\theta}^{+})
  \quad \text{and} \quad 1<p<\frac{2n}{n-2},\qquad\text{or} \\
  &\psi \in \Psi ^{1} _{n/2+\varepsilon}(S_{\theta}^{+})
  \quad \text{and} \quad \frac{2n}{n+2}<p<\infty.
\end{split}\end{equation*}
Then
\begin{equation*}
  \int_{\R^n} \Big( \iint_{|y-x|<t}
  |\psi(t\Delta)u(y)|^{2}\frac{\ud y\ud t}{t^{n+1}}\Big)^{p/2}\ud x \eqsim \|u\|_{L^{p}}^{p}.
\end{equation*}

\subsection*{Acknowledgments}
This paper was started while Tuomas Hyt\"onen and Jan van Neerven visited the Centre for Mathematics and its Applications (CMA) at
the Australian National University (ANU), and it was finished during Pierre Portal's visit to the Department of Mathematics and Statistics at the University of Helsinki.
Tuomas Hyt\"onen was supported by the Academy of Finland (SA) project 114374 ``Vector-valued singular integrals'', and by the CMA while in Canberra.
Jan van Neerven was supported by the VIDI subsidy 639.032.201 and VICI subsidy 639.033.604 of the Netherlands Organisation for Scientific Research (NWO).
Pierre Portal was supported by the CMA and the Australian Research Council as a postdoctoral fellow, and by the above-mentioned SA project while in Helsinki. He would like to thank Alan M$^{\rm c}$Intosh for his guidance.
The authors also wish to thank Alan M$^{\rm c}$Intosh for his kind hospitality at ANU, and for many discussions which motivated and influenced this work.

\section{Preliminaries}
\label{sect:prelim}

In this section we establish some terminology and collect
auxiliary results needed in the main body of the paper.

Let $X$ and $Y$ be Banach spaces and let $\calL(X,Y)$ denote the space of 
all bounded linear operators acting from $X$ into $Y$.
A family of bounded operators $\mathscr{T}\subseteq \calL(X,Y)$ is
called {\em $\g$-bounded} if there is a constant $C$ such that 
for all integers $k\ge 1$ and all 
$T_1,\dots,T_k\in \mathscr{T}$ and $\xi_1,\dots,\xi_j\in X$ we have
\beq\label{eq:g-bdd}
\E \Big\n \sum_{j=1}^k \g_j T_j \xi_j\Big\n^2 \le C^2
\E \Big\n \sum_{j=1}^k \g_j \xi_j\Big\n^2 .
\eeq
Here, $\g_1,\dots,\g_k$ are independent standard normal 
variables defined on some probability space $(\O,\F,\P)$ and $\E$ denotes
the expectation with respect to $\P$. 
The least admissible constant in \eqref{eq:g-bdd} is denoted by
$\g(\mathscr{T})$.

By the Kahane-Khintchine inequality,
the exponent $2$ may be replaced by any exponent $1\le p<\infty$
at the cost of a possibly different constant.

Upon replacing the standard normal variables by Rademacher variables
in \eqref{eq:g-bdd} one arrives
at the notion of $R$-boundedness. Every $R$-bounded family is $\g$-bounded, and
the converse holds if $Y$ has finite cotype. Since we are primarily
interested in UMD spaces $Y$, which have finite cotype, 
the distinction between $\g$-boundedness and $R$-boundedness is immaterial.
We prefer the former since our techniques are Gaussian and therefore the use of
Gaussian variables seems more natural. 

Let $H$ be a Hilbert space.  
A linear operator $R: H\to X$ is called {\em $\g$-summing}
if
$$ \n R\n_{\g^\infty(H,X)} := \sup \Big(\E \Big\n \sum_{j=1}^k \g_j
Rh_j\Big\n^2\Big)^\frac12 < \infty,$$
where the supremum is taken over all integers $k\ge 1$ and all finite orthonormal systems
$h_1,\dots,h_k$ in $H$.  
The space $\g^\infty(H,X)$, endowed with
the above norm, is a Banach space. The closed subspace of $\g^\infty(H,X)$
spanned by the finite rank operators is denoted by $\g(H,X)$. A linear operator
$R:H\to X$ is said to be {\em $\g$-radonifying} if it belongs to $\g(H,X)$.

A celebrated result of Hoffman-J{\o}rgensen and Kwapie\'n \cite{HJ, Kwap}
implies that $$ \g^\infty(H,X) = \g(H,X)$$
for Banach spaces $X$ not containing an isomorphic copy of $c_0$.

If $H$ is separable with orthonormal basis $(h_n)_{n\ge 1}$, then an operator $R:H\to X$
is $\g$-radonifying if and only if
the sum $\sum_{n\ge 1} \g_n Rh_n$ converges in $L^2(\O;X)$, in which case we
have $$ \n R\n_{\g(H,X)} =\Big(\E\Big\n \sum_{j\ge 1} \g_j Rh_j\Big\n^2\Big)^\frac12.$$

The following criterium for membership of $\g(H,X)$ will be referred to as {\em
covariance domination}.

\begin{proposition}\label{prop:CD}
Let $S\in \calL(H,X)$ and $T\in \g(H,X)$ satisfy
\[
 \n S\s \xi\s \n \le C \n T\s \xi\s\n, \quad \xi\s\in X\s,
\]
with $C$ independent of $\xi\s$. Then $S\in \g(H,X)$ and  $\n S\n_{\g(H,X)}\le C
\n T\n_{\g(H,X)}$.
\end{proposition}

For more details we refer to \cite{KW, NVW} and the references therein.

Let $(A,\Sigma,\mu)$ be a $\sigma$-finite measure space, $H$ a Hilbert spaces and $X$ a Banach space. 
In the formulation of the next result, which is a multiplier result due to
Kalton and Weis \cite{KW}, we identify $H \otimes X$-valued functions $f\otimes \xi$,\
where $f\in L^2(A,H)$ and $\xi\in X$, 
with the operator $R_{f\otimes
\xi}\in \g(L^2(A;H),X)$ defined by
\beq\label{eq:integral_operator}  
R_{f\otimes \xi} g := \langle f, g \rangle \otimes \xi, \quad g\in L^2(A;H).
\eeq
 where $\langle f,h\rangle$ denotes the scalar product on $L^2(A;H)$.

\begin{lemma}\label{lem:KW} Let $X$ be a Banach space, 
let $(A,\Sigma,\mu)$ be a $\sigma$-finite measure space,
and let $M:A\to \calL(X)$ be a function such that $a\mapsto M(a)\xi$
is strongly $\mu$-measurable for all $\xi\in X$. If the set 
$$\mathscr{M} = \{M(a): \ a\in A\}$$
is $\g$-bounded, then the mapping
$$ f(\cdot)\otimes \xi \mapsto f(\cdot)\otimes  M(\cdot)\xi, $$
extends to a bounded operator $M$ on $\g(L^2(A;H),X)$ of norm 
$\n M\n\le \gamma(\mathscr{M}).$ 
\end{lemma}

Let us also recall that for $1\le p<\infty$, the mapping
$ f\mapsto [h\mapsto f(\cdot)h]$
defines an isomorphism of Banach spaces
\beq\label{eq:fubini}
 L^p(A;\g(H,X))\eqsim \g(H,L^p(A;X)).
\eeq 
This follows from a
simple application of the Kahane-Khintchine inequality; we refer to  
\cite[Proposition 2.6]{NVW} for the details. Here, $H$ and $X$ are allowed
to be arbitrary Hilbert spaces and Banach spaces, respectively; the norm constants in the isomorphism are independent of $H$.\\

Let $\g = (\g_n)_{n\ge 1}$ be a sequence of independent standard normal
variables on a probability space $(\O,\F,\P)$. 
Recall that a Banach space $X$ is called {\em $K$-convex} if the mapping
$$\pi_\g: f \mapsto \sumn \g_n \E (\g_n f), \quad f\in L^2(\O;X),$$
defines a bounded operator on $L^2(\O;X)$. This notion is well-defined:
if $\pi_\g$ is bounded for some sequence $\g$, then it is
bounded for all sequences $\g$. A celebrated result of Pisier
\cite{Pi} states that $X$ is $K$-convex if and only if $X$ is $B$-convex if and
only if $X$ has nontrivial type.

If $X$ is $K$-convex, then the isometry
$ I_\g: \g(H,X) \to L^2(\O;X)$ defined by
$$ I_\g R := \sumn \g_n Rh_n$$
maps $\g(H,X)$ onto a complemented subspace of $L^2(\O;X)$.
Indeed, for all $R\in\g(H,X)$
we have
$$ \pi_\g I_\g R = \sumn \g_n \E \g_n\sumj \g_j Rh_j = \sumn \g_n Rh_n = I_\g
R.$$
Hence, the range of $I_\g$ is contained in the range of $\pi_\g$.
Since the range of $\pi_\g$ is spanned by the functions $\g_n \otimes \xi
= I_\g (h_n\otimes \xi)$, the range is $\pi_\g$ is contained in the range of
$I_\g$. We conclude that the ranges of $\pi_\g$ and $I_\g$ coincide and the
claim is proved. As an application of this we are able to describe complex
interpolation spaces of the spaces $\g(H,X)$.

\begin{proposition}\label{prop:interpolation}
 If $X_1$ and $X_2$ are $K$-convex, then for all $0<\theta<1$ we have
$$ [\g(H,X_1), \g(H,X_2)]_\theta = \g(H,[X_1,X_2]_\theta) \quad\hbox{with
equivalent norms}.$$  
\end{proposition}
\begin{proof}
In view of the preceding observations this follows from 
general results on interpolation of complemented subspaces 
\cite[Chapter 5]{BL}.
\end{proof}

\section{Main estimate}
\label{sect:mainest}
The main estimate of this paper is a $\g$-boundedness estimate for some averaging operators, which is proven below.

We start by recalling some known results.
The first is Bourgain's extension to UMD spaces of Stein's inequality \cite{Bou} 
(see \cite{CPSW} for a complete proof). 

\begin{lemma}\label{lem:stein} Let $1<p<\infty$ and let $X$ be a UMD space.
Let $(\F_m)_{m\in\Z}$ be a filtration on a probability space $(\O,\F,\P)$.
Then the family of conditional
expectations $$\mathscr{E} = \{\E(\,\cdot\,|\F_m): \ m\in\Z\}$$ is 
$\g$-bounded on $L^p(\O;X)$.
 \end{lemma}

Let us agree that a cube in $\R^n$ is any set $Q$ of the form $x+[0,\ell)^n$ with $x\in\R^n$ and $\ell>0$. We denote $\ell(Q):=\ell$ and call it the side-length of $Q$. A \emph{system of dyadic cubes} is a collection $\Delta=\bigcup_{k\in\Z}\Delta_{2^k}$, where $\Delta_{2^k}$ is a disjoint cover of $\R^n$ by cubes of side-length $2^k$, and each $Q\in\Delta_{2^k}$ is the union of $2^n$ cubes $R\in\Delta_{2^{k-1}}$. We recall the following geometric 
lemma of Mei~\cite{Mei}:

\begin{lemma}
There exist $n+1$ systems of dyadic cubes $\Delta^0,\ldots,\Delta^n$ and a constant $C<\infty$ such that for any ball $B\subset\R^n$ there is a $Q\in\bigcup_{k=0}^n\Delta^k$ which satisfies $B\subset Q$ and $\abs{Q}\leq C\abs{B}$.
\end{lemma}

The following results can be found in \cite{HMP}:

\begin{lemma}\label{lem:Figiel}
Let $X$ be a UMD space and $1<p<\infty$. Let $r\in\Z^n\setminus\{0\}$ and $x_Q\in X$ for all $Q\in\Delta$. Then
\[
  \E\Big\|\sum_{k\in\Z}\e_k\sum_{Q\in\Delta_{2^k}} 1_{Q+r\ell(Q)} x_Q\Big\|_p
  \leq C(1+\log\abs{r})\E\Big\|\sum_{k\in\Z}\e_k\sum_{Q\in\Delta_{2^k}} 1_{Q} x_Q\Big\|_p.
\]
\end{lemma}

\begin{lemma}\label{lem:jump}
Let $X$ be a UMD space, $1<p<\infty$, and $m\in\Z_+$. For each $Q\in\Delta$, let $Q',Q''\in\Delta$ be subcubes of $Q$ of side-length $2^{-m}Q$. Then for all $\ell\in\Z$ and all $x_Q\in X$
\[
  \E\Big\|\sum_{k\equiv\ell}\e_k \sum_{Q\in\Delta_{2^k}}1_{Q''} x_Q\Big\|_p
  \leq C\E\Big\|\sum_{k\equiv\ell}\e_k \sum_{Q\in\Delta_{2^k}}1_{Q'} x_Q\Big\|_p,
\]
where $k\equiv\ell$ is short-hand for $k\equiv\ell\mod (m+1)$.
\end{lemma}

The previous lemmas will now be used to prove our main estimate.

\begin{proposition}\label{prop:averaging}
Let $X$ be a UMD space, $1<p<\infty$, and let $L^p(X)$ have type $\tau$. For $\alpha\geq 1$, let $\mathscr{A}_{\alpha}$ be the family of operators
\begin{equation*}
  f\mapsto A^{\alpha}_B f:=1_{\alpha B}\intav_B f \ud x,
\end{equation*}
where $B$ runs over all balls in $\R^n$. Then $\mathscr{A}_{\alpha}$ is $\gamma$-bounded on $L^p(X)$ with the $\gamma$-bound at most $C(1+\log\alpha)\alpha^{n/\tau}$ and $C$ depends only on $X$, $p$, $\tau$ and $n$.
\end{proposition}

\begin{proof}
We have to show that
\[
  \E\Big\|\sum_{j=1}^k\e_j 1_{\alpha B_j}\intav_{B_j} f_j\ud x\Big\|_p
  \leq C\E\Big\|\sum_{j=1}^k\e_j f_j\Big\|_p.
\]
By splitting all the balls $B_j$ into $n+1$ subsets and considering each of them separately, we may assume by Mei's lemma that there is a system of dyadic cubes $\Delta$ and $Q_1,\ldots,Q_k\in\Delta$ such that $B_j\subset Q_j$ and $\abs{Q_j}\leq C\abs{B_j}$.

Let $m$ be the integer for which $2^{m-1}\leq\alpha<2^m$. Let $Q_j^*\in\Delta$ be the unique cube in the dyadic system which has side-length $2^m\ell(Q_j)$ and contains $Q_j$. Then $\alpha B_j$ is contained in the union of $Q_j^*$ and at most $2^n-1$ of adjacent cubes $R\in\Delta$ of the same size. Writing $g_j=1_{B_j} f_j$, we observe that
\[
  \intav_{B_j} f_j\ud x=\frac{\abs{Q_j}}{\abs{B_j}}\intav_{Q_j} g_j\ud x.
\]
Since $\abs{Q_j}/\abs{B_j}\leq C$, by the contraction principle it suffices to show that
\[
  \E\Big\|\sum_{j=1}^k\e_j 1_{R_j}\intav_{Q_j}g_j\ud x\Big\|_p
  \leq C\E\Big\|\sum_{j=1}^k\e_j g_j\Big\|_p,
\]
where $R_j=Q_j^*+r\ell(Q_j^*)$ for some $\abs{r}\leq n$. Thanks to Lemma~\ref{lem:Figiel}, it suffices to consider $r=0$.

We next write $Q_j^*$ as the union $\bigcup_{i=1}^{M}Q_{ji}$, where $Q_{ji}\in\Delta$ are the $M:=2^{nm}$ subcubes of $Q_j^*$ of side-length $\ell(Q_j)$. Let us fix the enumeration so that $Q_{j1}=Q_j$. Writing $x_j:=\intav_{Q_j} g_j\ud x$ for short, it follows that
\[\begin{split}
  \E\Big\|\sum_{j=1}^k\e_j 1_{Q_j^*} x_j\Big\|_p
  &=\E\Big\|\sum_{i=1}^M\sum_{j=1}^k\e_j 1_{Q_{ji}} x_j\Big\|_p
  \leq C\E'\E\Big\|\sum_{i=1}^M\e_i'\sum_{j=1}^k\e_j 1_{Q_{ji}} x_j\Big\|_p \\
  &\leq C\Big(\sum_{i=1}^M\E\Big\|\sum_{j=1}^k\e_j 1_{Q_{ji}} x_j\Big\|_p^{\tau}\Big)^{1/\tau}
\end{split}\]
where the first estimate follows from the Khintchine--Kahane inequality and the disjointness of the $Q_{ji}$ for each fixed $j$, and the second from the assumed type-$\tau$ property.

If we assume, for the moment, that all the side-lengths $2^{k(j)}:=\ell(Q_j)$ satisfy $k(j)\equiv k(j')\mod (m+1)$, we may apply Lemma~\ref{lem:jump} to continue the estimate with
\[\begin{split}
 \leq C\Big(\sum_{i=1}^M\E\Big\|\sum_{j=1}^k\e_j 1_{Q_j} x_j\Big\|_p^{\tau}\Big)^{1/\tau}
   &\leq C M^{1/\tau}\E\Big\|\sum_{j=1}^k\e_j 1_{Q_j}\intav_{Q_j} g_j\ud x\Big\|_p \\
  &\leq CM^{1/\tau}\E\Big\|\sum_{j=1}^k\e_j g_j\Big\|_p,
\end{split}\]
where the last estimate applied Stein's inequality, observing that the operators $g\mapsto 1_{Q_j}\intav_{Q_j}g \ud x$ are conditional expectations related to the dyadic filtration induced by $\Delta$. Since $M=2^{nm}\leq 2^n\alpha^n$, we obtain the assertion even without the logarithmic factor in this case.

In general, the above assumption may not be satisfied, but we can always split the indices $j$ into $m+1\leq c(1+\log\alpha)$ subsets which verify the assumption, and this concludes the proof.
\end{proof}

\begin{remark}
The proof simplies considerably in the important special case $\alpha=1$.
\end{remark}

\section{The vector-valued tent spaces $T^{p,2}(X)$}
\label{sect:tent}

In order to motivate our approach we begin with a simple characterisation of
tent spaces in the scalar case. 
We put $\R_+^{n+1} := \R^n\times \R_+$ and denote
$$\Gamma(x) = \{(y,t)\in \R_+^{n+1}: \ |x-y|<t\}.$$
Thus $(y,t)\in \Gamma(x) \Leftrightarrow y\in B(x,t)$,
where $B(x,t) = \{y\in \R^n: \ |x-y|<t\}$. 
We shall write $$L^p=L^p(\R^n), \qquad L^2(\frac{\ud y\, \ud t}{t^{n+1}}) = L^2\big(\R_+^{n+1},\frac{\ud y\, \ud t}{t^{n+1}}\big),$$ 
where $\ud y$ and $\ud t$ denote the Lebesgue measures on $\R^n$ and $\R_+$. 
Similar conventions will apply to their vector-valued analogues.  
The dimension $n\ge 1$ is considered
to be fixed. 

For $1\le p,q<\infty$, the  {\em tent space} $ T^{p,q}= T^{p,q}(\R_+^{n+1})$ 
consists of all (equivalence classes of) measurable functions 
$f:\R_+^{n+1} \to \C$ with the property that
$$\int_{\R^n}\Big(\int_{\Gamma(x)} |f(y,t)|^q
\frac{\ud y\ud t}{t^{n+1}}\Big)^\frac{p}{q} \ud x$$
is finite. With respect to the norm
$$ \n f\n_{T^{p,q}(\R_+^{n+1})} := \Big\n  \Big(\int_{\Gamma(\cdot)} |f(y,t)|^q
\frac{\ud y\ud t}{t^{n+1}}\Big)^\frac1q\Big\n_{L^p},$$
$T^{p,q}$ is a Banach space.
Tent spaces were introduced in the 1980's by Coifman, Meyer, and Stein
\cite{CMS}. Some of the principal results of that paper were simplified 
by Harboure, Torrea, and Viviani \cite{HTV}, who exploited the fact that 
\[
J: f\mapsto \big[x\mapsto [(y,t)\mapsto 1_{B(x,t)}(y)f(y,t)]\big]
\]
maps $T^{p,q}$ isometrically onto
a complemented subspace of $L^p(L^q(\frac{\ud y\ud t}{t^{n+1}}))$
for $1<p,q<\infty$.
 
We now take $q=2$, $H$ a Hilbert space, and extend the mapping $J$ to functions 
in $C_c(H)\otimes X$ by
$J(g\otimes \xi) := Jg \otimes \xi$ and linearity. Here, $C_c(H)$ denotes the 
space of $H$-valued continuous functions on $\R_+^{n+1}$ with compact support.
Note that by \eqref{eq:integral_operator}, $J(g\otimes \xi)$ defines an element of 
$L^p(\g(L^2(\frac{\ud y\ud t}{t^{n+1}};H),X))$ in a natural way.

\begin{definition} Let $1\leq  p<\infty$.
The {\em tent space} $T^{p,2}(H;X)$ is defined as
the completion of $C_c(H)\otimes X$ with respect to the norm
$$ \n f\n_{T^{p,2}(H;X)} 
:= \n J f\n_{L^p(\g(L^2(\frac{\ud y\ud t}{t^{n+1}};H),X))}.$$
$T^{p,2}(\C;X)$  will simply be denoted by $T^{p,2}(X)$.
\end{definition}

It is immediate from this definition that $J$ defines an isometry from
$T^{p,2}(H;X)$ onto a closed subspace of 
$L^p(\g(L^2(\tfrac{\ud y\ud t}{t^{n+1}};H),X))$. In what follows we shall
always identify $T^{p,2}(H;X)$ with its image in
$L^p(\g(L^2(\tfrac{\ud y\ud t}{t^{n+1}};H),X)$.

Using the identification $\g(L^2(\tfrac{\ud y\ud t}{t^{n+1}}),\C) = L^2(\tfrac{\ud y\,
dt}{t^{n+1}})$ we see that our definition extends the definition
of tent spaces in the scalar-valued case.

Our first objective is to prove that 
if $X$ is a UMD space, then $T^{p,2}(H;X)$ is  complemented in 
$L^p(\g(L^2(\tfrac{\ud y\ud t}{t^{n+1}};H),X))$.

\begin{proposition}\label{prop:complemented}
Let $1<p<\infty $, $H$ a Hilbert space, and $X$ a UMD space. The mapping 
$$N f(x,y,t):= \frac{ 1_{B(y,t)}(x)}{|B(y,t)|}\int_{B(y,t)} f(z,y,t)\ud z,$$
initially defined for operators of the form \eqref{eq:integral_operator},
extends to a bounded projection in
\begin{equation*} 
  L^p(\g(L^2(\tfrac{\ud y\ud t}{t^{n+1}};H),X))
\end{equation*}
whose range is \(T^{p,2}(H;X)\).
\end{proposition}

\begin{proof}
We follow the proof of Harboure, Torrea, and Viviani \cite[Theorem 2.1]{HTV} for
the scalar-valued case,
the main difference being that the use of maximal functions is replaced
by a $\g$-boundedness argument using averaging operators.

First we prove that $N$ is a bounded operator.
In view of the isomorphism 
\eqref{eq:fubini}
it suffices to prove that $N$ acts as a bounded operator on 
$\g(L^2(\tfrac{\ud y\ud t}{t^{n+1}};H),L^p(X))$.
This will be achieved 
by identifying $N$ as a pointwise multiplier
on $L^p(X)$ with $\g$-bounded range, and then 
applying Lemma \ref{lem:KW}.
In fact, putting
$$ N(y,t)\,g:=  \frac{1_{B(y,t)}}{|B(y,t)|}\int_{B(y,t)} g(z)\ud z, \quad
g\in L^p(X),$$
and $f_{y,t}(x) := f(x,y,t):=\widetilde{f}(y,t)\otimes g(x)$, we have
$$Nf(\cdot,y,t)= \widetilde{f}(y,t)\otimes N(y,t)g = \widetilde{f}(y,t) \otimes A_{B(y,t)}g.$$
The $\g$-boundedness of $\{N(y,t): \ (y,t)\in \R_+^{n+1}\}$
now follows from Proposition \ref{prop:averaging}.

Knowing that $N$ is bounded on $L^p(\g(L^2(\tfrac{\ud y\ud t}{t^{n+1}};H),X))$,
the fact that it is a projection follows from the scalar case, noting that
the linear span of the 
functions of the form $1_{B(x,t)}\otimes (f\otimes \xi)$,
with $f\in C_c(H)$, $x\in \R^n$, and $t>0$, is dense in
$L^p(\g(L^2(\tfrac{\ud y\ud t}{t^{n+1}};H),X))$.   
\end{proof} 

For $\alpha>0$ the vector-valued tent space $T_{\alpha}^{p,2}(H;X)$ may be 
defined as above in terms of the norm
$$ \n f\n_{T_{\alpha}^{p,2}(H;X)} := \| J_{\alpha} f\|_{
L^p(\g(L^2(\frac{\ud y\ud t}{t^{n+1}};H),X))},$$
where $J_{\alpha} f := \big[x\mapsto [(y,t)\mapsto 1_{B(x,\alpha t)}(y)f(y,t)]\big]$.

\begin{theorem}\label{thm:angle}
Let $1<p<\infty$,  $H$ a Hilbert space and $X$ a UMD space such that $L^p(H \otimes X)$ has type $\tau$. 
For all $\alpha>0$, 
a strongly measurable function $f:\R_+^{n+1}\to H \otimes X$ belongs to
$T^{p,2}(H;X)$ if and only if it belongs to $T_{\alpha}^{p,2}(H;X)$. Moreover,
there exists a constant $C = C(p,X)$ 
such that
\begin{equation}\label{eq:angle} 
   \n f\n_{T^{p,2}(H;X)}\le \n f\n_{T_{\alpha}^{p,2}(H;X)}\le C(1+\log\alpha)\alpha^{n/\tau}
   \n f\n_{T^{p,2}(H;X)}
\end{equation}
for $f\in T^{p,2}(H;X)$ and $\alpha>1$.
\end{theorem}

\begin{proof}
It suffices to prove the latter estimate in~\eqref{eq:angle}. On $L^p(\g(L^2(\dydt;H),X))$, we consider the operator
\[
  N_{\alpha} f(x,y,t): = \frac{1_{B(y,\alpha t)}(x) }{\abs{B(y,t)}}\int_{B(y,t)}f(z,y,t)\ud z.
\]
Simple algebra shows that $N_{\alpha}Jf=J_{\alpha}f$, and hence
\[\begin{split}
  \|f\|_{T^{p,2}_{\alpha}(X)}
  &=\|J_{\alpha}f\|_{L^p(\gamma(L^2(\dydt;H),X))}
  =\|N_{\alpha} Jf\|_{L^p(\gamma(L^2(\dydt;H),X))} \\
  &\leq \|N_{\alpha}\|_{\calL(L^p(\gamma(L^2(\dydt;H),X)))} \|Jf\|_{L^p(\gamma(L^2(\dydt;H),X))}.
\end{split}\]
By the isomorphism \eqref{eq:fubini}, we may consider the boundedness of $N_{\alpha}$ on the space 
$\gamma(L^2(\dydt;H),L^p(X))$ instead, and here this operator acts as the pointwise multiplier
$$N_{\alpha}(\widetilde{f}\otimes g)(\cdot,y,t)=\widetilde{f}(y,t) \otimes A^{\alpha}_{B(y,t)}g.$$
So, its boundedness with the asserted estimate follows from Proposition~\ref{prop:averaging}.
\end{proof}

\begin{remark}
If \(X=\C\), then one can take \(\tau=\min(2,p)\) in Theorem~\ref{thm:angle}. Except possibly for the logarithmic factor, \eqref{eq:angle} gives the correct order of growth of \(\|f\|_{T^{p,2}_{\alpha}}\) in terms of the angle \(\alpha\geq 1\).

To see this, consider functions of the form \(f(y,t)=1_{[1,2]}(t)g(y)\). Then
\begin{equation*}
  \|f\|_{T^{p,2}_{\alpha}}
  =\big\|(\eta_{\alpha}*\abs{g}^2)^{1/2}\big\|_p,
\end{equation*}
where the \(\eta_{\alpha}\) are functions having pointwise bounds \(c 1_{B(0,\alpha)}\leq\eta_{\alpha}\leq C 1_{B(0,C\alpha)}\) for some constants \(C>1>c>0\) depending only on \(n\).

Let us take \(g=\abs{g}^2=1_{B(0,1)}\). Then \((\eta_{\alpha}*\abs{g}^2)^{1/2}=\tilde\eta_{\alpha}\), where \(\tilde\eta_{\alpha}\) is another similar function, and hence
\begin{equation*}
  \|f\|_{T^{p,2}_{\alpha}}
  =\big\|(\tilde\eta_{\alpha})^{1/2}\big\|_{p}
  \eqsim\alpha^{n/p}
  \eqsim\alpha^{n/p}\|f\|_{T^{p,2}}.
\end{equation*}
This proves the sharpness for \(p\leq 2\).

Let us then choose \(g=g_{\alpha}=1_{B(0,\alpha)}\). Then
\begin{equation*}
  \eta_{\alpha}*\abs{g_{\alpha}}^2=\alpha^n\overline\eta_{\alpha},\qquad
  \eta_{1}*\abs{g_{\alpha}}^2=\underline\eta_{\alpha},
\end{equation*}
where \(\overline\eta_{\alpha},\underline\eta_{\alpha}\) are yet more similar functions as \(\eta_{\alpha}\). Writing \(f_{\alpha}(y,t)=1_{[1,2]}(t)g_{\alpha}(y)\), we have
\begin{equation*}
  \|f_{\alpha}\|_{T^{p,2}_{\alpha}}
  =\big\|(\alpha^n\overline\eta_{\alpha})^{1/2}\big\|_p
  =\alpha^{n/2}\big\|(\overline\eta_{\alpha})^{1/2}\big\|_p
  \eqsim\alpha^{n/2}\big\|(\underline\eta_{\alpha})^{1/2}\big\|_p
  =\alpha^{n/2}\|f_{\alpha}\|_{T^{p,2}}.
\end{equation*}
This proves the sharpness for \(p\geq 2\).

In fact, for \(p=2\), a simple application of Fubini's theorem shows that we have the equality \(\|f\|_{T^{2,2}_{\alpha}}=\alpha^{n/2}\|f\|_{T^{2,2}}\) for all \(f\in T^{2,2}\) and \(\alpha>0\), so the logarithmic factor is unnecessary in this case.
\end{remark}

Sometimes it is useful to use tent space norms defined with a smooth cut-off instead of the sharp cut-off 
\(1_{B(x,t)}(y)\).
Given a function \(\phi \in C_c^{\infty}(\R)\) such that \(\phi(w)=1 \) if \(|w| \leq \frac{1}{2}\) and \(\phi(w)=0 \) if \(|w| \geq 1\), 
we are thus led to consider the mapping 
$J_{\phi} f := \big[x\mapsto [(y,t)\mapsto \phi(\frac{|y-x|}{t})f(y,t)]\big]$
and $$\|f\|_{T^{p,2}_{\phi}(H;X)}:=
\|J_{\phi}f\|_{L^p(\g(L^2(\tfrac{\ud y\ud t}{t^{n+1}};H),X))}.$$
\begin{proposition}
\label{prop:smooth}
Let $1<p<\infty$,  $H$ a Hilbert space and $X$ a UMD space.
A strongly measurable function $f:\R_+^{n+1}\to H \otimes X$ belongs to
$T^{p,2}(H;X)$ if and only if it belongs to $T_{\phi}^{p,2}(H;X)$. Moreover,
$$
 \n f\n_{T_{\phi}^{p,2}(H;X)}\eqsim   \n f\n_{T^{p,2}(H;X)}
$$
for $f\in T^{p,2}(H;X)$.
\end{proposition}
\begin{proof}
The proof is the same as that of Theorem \ref{thm:angle}.
Consider the operators
\begin{equation*}\begin{split}
  N_{\phi} f(x,y,t) &:= \frac{\phi(\frac{|y-x|}{t})}{\abs{B(y,t)}}\int_{B(y,t)}f(z,y,t)\ud z,\\
  \widetilde{N}_{\frac{1}{2}} f(x,y,t) &: = \frac{1_{B(x,\frac{t}{2})}}{\abs{B(y,\frac{t}{2})}}\int_{B(y,\frac{t}{2})}f(z,y,t)\ud z.
\end{split}\end{equation*} 
We have \(J_{\phi} = N_{\phi}J\) and \(J_{\frac{1}{2}} = \widetilde{N}_{\frac{1}{2}}J_{\phi}\).
Moreover the operators \(N_{\phi}\) and \( \widetilde{N}_{\frac{1}{2}}\) act as the pointwise multipliers
\begin{equation*}\begin{split}
  N_{\phi}(\widetilde{f}\otimes g)(\cdot,y,t)
    &=\widetilde{f}(y,t) \otimes M^{\phi}_{y,t}A^{1}_{B(y,t)}g, \\
  \widetilde{N}_{\frac{1}{2}}(\widetilde{f}\otimes g)(\cdot,y,t)
    &=\widetilde{f}(y,t) \otimes A^{1}_{B(y,\frac{t}{2})}g.
\end{split}\end{equation*} 
where \(M^{\phi}_{y,t}g(x):=\phi(\frac{|y-x|}{t})g(x)\).
By  Lemma \ref{lem:KW} and Theorem \ref{thm:angle} the result follows from Proposition~\ref{prop:averaging} 
and Kahane's contraction principle.
\end{proof}

If $X$ is a UMD space, $H$ a Hilbert space, and $1<p,q<\infty$ satisfy 
$\frac1p+\frac1q=1$, we have natural isomorphisms
$$
\bal
\ & (L^p(\g(L^2(\tfrac{\ud y\ud t}{t^{n+1}};H),X)))\s 
\\ & \qquad  \eqsim
L^q((\g(L^2(\tfrac{\ud y\ud t}{t^{n+1}};H),X))\s)
\eqsim
L^q(\g(L^2(\tfrac{\ud y\ud t}{t^{n+1}};H),X\s))).
\eal
$$
The first of these follows from the fact that
$X$, and therefore $\g(L^2(\tfrac{\ud y\ud t}{t^{n+1}};H),X)$,
is reflexive, and the second follows from the $K$-convexity of UMD spaces.
Denoting by $N$ the projection of Proposition 
\ref{prop:complemented}, it is easily verified that
under the above identification the adjoint $N\s$ is given by the same formula.
As a result we obtain the following representation for the dual of $T^{p,2}(H;X)$:

\begin{theorem}
If $X$ is a UMD space, $H$ a Hilbert space, and $1<p,q<\infty$ satisfy 
$\frac1p+\frac1q=1$, we have a natural isomorphism
$$ (T^{p,2}(H;X))\s \eqsim T^{q,2}(H;X\s).$$
\end{theorem} 

As an immediate consequence of Proposition
\ref{prop:interpolation} we obtain the following result.

\begin{theorem}
\label{thm:Tp-interpolation}
Let $1<p_0\le p_1<\infty$, $H$ a Hilbert space, and let $X_0$ and $X_1$ be UMD spaces. Then for all
$0<\theta<1$ we have
$$ [T^{p_0,2}(H;X_0), T^{p_1,2}(H;X_1)]_\theta = T^{p_\theta,2}(H;[X_0,
X_1]_\theta), \quad \frac{1}{p_\theta} = \frac{1-\theta}{p_0}+\frac{\theta}{p_1}.$$ 
\end{theorem}
\begin{proof}
The result follows by combining \eqref{eq:fubini} with the following
facts: (i) 
if $X$ is a UMD space, then $L^p(X)$ is a UMD space for all $1<p<\infty$, 
(ii) UMD spaces
are $K$-convex, (iii) for $1\le p_0\le p_1<\infty$ we have 
$[L^{p_0}(X_0), L^{p_1}(X_1)]_\theta = L^{p_\theta}([X_0,X_1]_\theta)$ with 
$p_\theta$ as above.
\end{proof}

We conclude this section with a result showing that certain singular integral operators are bounded from $L^{p}(X)$ to
$T^{p,2}(X)$. This gives a Banach space-valued extension of \cite[Section 4]{HTV}.

\begin{theorem}
\label{thm:sing}
Let \(X\) be a UMD space. Consider the 
singular integral operator defined by
\begin{equation*} 
  Sf(t,y) = \int  _{\R^{n}} k_{t}(y,z)f(z)\ud z 
\end{equation*} 
for \(f \in C_{c}(\R^{n})\) and a measurable complex-valued function \((t,y,z)\mapsto k_{t}(y,z)\).
Assume that
\begin{enumerate}
\item[\rm (1)]
\(S \in \calL(L^{2},T^{2,2})\),
\item[\rm (2)]
There exists \(\alpha>0\) such that for all $y,z \in \R^{n}$ and $t>0$ we have
\begin{equation*}
  |k_{t}(y,z)| \lesssim \frac{t^{\alpha}}{(|y-z|+t)^{n+\alpha}},
\end{equation*}
\item[\rm (3)]
There exists $\beta>0$ such that for all $t>0$ and all 
$ y,z,z' \in \R^{n}$ satisfying \(|z-y|+t >2|z-z'|\) we have 
\begin{equation*}
  |k_{t}(y,z)-k_{t}(y,z')| \lesssim
  \frac{t^{\beta}|z-z'|}{(|y-z|+t)^{n+1+\beta}},
\end{equation*}
\item[\rm (4)] For all $t>0$ and \(y\in\R^n\) we have
\begin{equation*}
  \int _{\R^{n}}k_{t}(y,z)\ud z = 0.
\end{equation*}
\end{enumerate}
Let \(1<p<\infty\). Then $S\otimes I_X$ extends to a bounded operator from
\(L^{p}(X)\) to \(T^{p,2}(X)\).
\end{theorem}

\begin{proof}
We consider the auxiliary operator \(T\) taking \(X\)-valued functions 
to ones with values in \(\gamma(L^2(\dydt),X)\), 
given by
\begin{equation*} 
  Tf(x) = \int  _{\R^{n}} K(x,z)\otimes f(z)\ud z,\qquad f\in C_c(X),
\end{equation*} 
where \(K(x,z)\) is the \(L^2(\dydt)\)-valued kernel defined by
\begin{equation*}
  K(x,z) : (y,t)\mapsto
  \phi\big(\frac{|y-x|}{t}\big)k_{t}(y,z)
\end{equation*}
for some even \(\phi \in C_c^{\infty}(\R)\) such that \(\phi(w)=1 \) if \(|w| \leq \frac{1}{2}\), \(\phi(w)=0 \) if \(|w| \geq 1\), and \(\int_0^1\phi(r)r^{n-1}\ud r=0\).
The claim of the theorem follows if we can show that \(T\) extends to a bounded operator from \(L^{p}(X)\) to
\(L^{p}(\gamma(L^{2}(\dydt);X))\).
This is proved by applying a version of the \(T(1)\) theorem for Hilbert space -valued kernels from~\cite{HK} (which, in turn, is based on results from \cite{hw-t1,kw-wavelet}).  We first remark that the condition \(T(1)=0\) follows directly from (4), whereas the vanishing integral assumption on \(\phi\) guarantees that \(T'(1)=0\), too.
It remains to check the following \(L^2(\dydt)\)-valued versions of the standard estimates:
\begin{equation}
\label{t1-claim1}
\underset{x,z \in \R^{n}}{\sup} |x-z|^{n}
  \|K(x,z)\|_{L^2(\dydt)}\lesssim 1,
\end{equation}
\begin{equation}
\label{t1-claim2}
  \underset{\underset{|x-z|>2|x-x'|}{x,x',z \in \R^{n}}}{\sup} 
  \frac{|x-z|^{n+1}}{|x-x'|}
  \|K(x,z)-K(x',z)\|_{L^2(\dydt)}\lesssim 1,
\end{equation}
\begin{equation}
\label{t1-claim3}
  \underset{\underset{|x-z|>2|z-z'|}{x,z,z' \in \R^{n}}}{\sup} 
  \frac{|x-z|^{n+1}}{|z-z'|}
  \|K(x,z)-K(x,z')\|_{L^2(\dydt)}\lesssim 1,
\end{equation}
and the weak boundedness property: for any \(\eta, \widetilde{\eta}\in C^{\infty}_c(B(0,1))\) which satisfy the bounds \(\|\eta\|_\infty,\|\widetilde\eta\|_\infty,\|\nabla\eta\|_\infty,\|\nabla\widetilde\eta\|_\infty\leq 1\), one should have
\begin{equation}
\label{t1-claim4}
  \underset{(u,r) \in \R^{n}\times\R_{+}}{\sup}  
  \Big\|\int_{\R^{n}}\!\!  \int_{\R^{n}} \!
  K(x,z)\eta\big(\frac{x-u}{r}\big)\widetilde{\eta}(\frac{z-u}{r})
  \frac{\ud z\ud x}{r^{n}}\Big\|_{L^2(\dydt)} \lesssim 1.
\end{equation}

{\it Proof of (\ref{t1-claim1})}:
Using (2) and noting that 
we have $\phi\big(\frac{|y-x|}{t}\big) =0$ for $y\not\in B(x,t)$,
\begin{equation*}\begin{split}
  &\int  _{0} ^{\infty} \int  _{\R^{n}}
   \Big|\phi\big(\frac{|y-x|}{t}\big)k_{t}(y,z)\Big|^{2}\frac{\ud y\ud t}{t^{n+1}}\\
  &\lesssim \int  _{0} ^{|x-z|} \int  _{B(x,t)} \Big|
    \frac{t^{\alpha}}{(|x-z|+t-|y-x|)^{n+\alpha}}\Big|^{2}\frac{\ud y\ud t}{t^{n+1}}
 + \int  _{|x-z|} ^{\infty} \int_{B(x,t)}\frac{\ud y\ud t}{t^{3n+1}}\\
  &\lesssim \int  _{0} ^{|x-z|} \frac{t^{2\alpha-1}}{|x-z|^{2n+2\alpha}}\ud t
   + \int  _{|x-z|} ^{\infty} \frac{\ud t}{t^{2n+1}} \lesssim |x-z|^{-2n}.
\end{split}\end{equation*}
{\it Proof of (\ref{t1-claim2})}: Using (2) and
the mean value theorem and reasoning as above,  for $x,x',z$ satisfying $|x-z|> 2|x-x'|$  we have 
\begin{equation*}\begin{split}
  \int  _{0} ^{\infty} \int  _{\R^{n}}
  &    \Big|\Big(\phi\big(\frac{|y-x|}{t}\big)-\phi\big(\frac{|y-x'|}{t}\big)\Big)k_{t}(y,z)\Big|^{2}
    \frac{\ud y\ud t}{t^{n+1}} \\
 &\lesssim \int  _{0} ^{\infty} \int  _{B(x,t)} 
    \Big(\frac{|x-x'|t^{\alpha}}{t(|y-z|+t)^{n+\alpha}}\Big)^{2}
    \frac{\ud y\ud t}{t^{n+1}} + \text{similar} \\
 & \lesssim \int  _{0} ^{|x-z|} \int  _{B(x,t)} 
  \Big(\frac{|x-x'|t^{\alpha}}{t(|x-z|+t-|y-x|)^{n+\alpha}}\Big)^{2}
  \frac{\ud y\ud t}{t^{n+1}} \\
 & \qquad + \int  _{|x-z|} ^{\infty} |x-x'|^{2}\frac{\ud t}{t^{2n+3}} + \text{similar} \\
 & \lesssim \int  _{0} ^{|x-z|} \frac{t^{2\alpha-3}|x-x'|^{2}}{|x-z|^{2n+2\alpha}}dt
   + \frac{|x-x'|^{2}}{|x-z|^{2n+2}} + \text{similar} \\
 &\lesssim \frac{|x-x'|^{2}}{|x-z|^{2n+2}},
\end{split}\end{equation*}
where the words ``similar'' above refer to a copy of the other terms appearing in the same step, with all the occurences of \(x\) and \(x'\) interchanged.

{\it Proof of (\ref{t1-claim3})}: Using (3),  for $x,z,z'$ satisfying $|x-z|>
2|z-z'|$ we have 
\begin{equation*}\begin{split}
  &\int  _{0} ^{\infty} \int  _{\R^{n}} 
  \Big|\phi\big(\frac{|y-x|}{t}\big)\big(k_{t}(y,z)-k_{t}(y,z')\big)\Big|^{2}\frac{\ud y\ud t}{t^{n+1}}\\
  & \lesssim  \int  _{0} ^{\infty} \int  _{B(x,t)}
  \Big(\frac{t^{\beta}|z-z'|}{(|z-y|+t)^{n+1+\beta}}\Big)^{2}\frac{\ud y\ud t}{t^{n+1}} \\
  & \lesssim  \int  _{0} ^{|x-z|} \int  _{B(x,t)}
  \Big(\frac{t^{\beta}|z-z'|}{(|z-x|+t-|y-x|)^{n+1+\beta}}\Big)^{2}\frac{\ud y\ud t}{t^{n+1}} 
    + \int  _{|x-z|} ^{\infty} \frac{|z-z'|^{2}}{t^{2n+3}}\ud t \\
  & \lesssim  \int  _{0} ^{|x-z|}
  \frac{t^{2\beta-1}|z-z'|^{2}}{|z-x|^{2n+2+2\beta}}\ud t
  + \int  _{|x-z|} ^{\infty} \frac{|z-z'|^{2}}{t^{2n+3}}\ud t \lesssim \frac{|z-z'|^{2}}{|x-z|^{2n+2}}.
\end{split}\end{equation*}

{\it Proof of (\ref{t1-claim4})}: Using the Cauchy-Schwarz inequality and (1) 
we have 
\begin{equation*}\begin{split}
   \int  _{0} ^{\infty} \int  _{\R^{n}} &
  \Big|\int _{\R^{n}} \int _{\R^{n}} 
  \phi\big(\frac{|y-x|}{t}\big)k_{t}(y,z)\eta\big(\frac{x-u}{r}\big)
  \widetilde{\eta}\big(\frac{z-u}{r}\big)\frac{\ud z\ud x}{r^{n}}\Big|^{2}
  \frac{\ud y\ud t}{t^{n+1}}\\
  & \lesssim \frac1{r^{n}} \int  _{0} ^{\infty} \int  _{\R^{n}} \int  _{\R^{n}} 
  \Big|\phi\big(\frac{|y-x|}{t}\big)\int  _{\R^{n}}
  k_{t}(y,z)\widetilde{\eta}\big(\frac{z-u}{r}\big)\ud z\Big|^{2}\frac{\ud y\ud t\ud x}{t^{n+1}} \\
  & \lesssim
  \frac1{r^n}\big\|S\big(\widetilde{\eta}\big(\frac{\cdot-u}{r}\big)\big)\big\|^{2} _{T^{2,2}} 
  \lesssim \|\widetilde{\eta}\|^{2}_{L^{2}} \lesssim 1.
\end{split}\end{equation*}
This concludes the proof.
\end{proof}

\section{Off-diagonal estimates and their consequences}
\label{sect:od}

We start by recalling some terminology.
\begin{definition}
Let $M,t>0$ and $H$ a Hilbert space. An operator $T \in \mathcal{L}(L^{2}(\R^{n},H))$ is said to have 
{\em off-diagonal estimates of order $M$ at the scale of $t$}
 if there is a constant $C$ such that 
\begin{equation*}
  \n T f\n_{L^2(E;H)} \leq C \lb d(E,F)/t\rb^{-M}\n f\n_{L^2(F;H)}
\end{equation*}
for all Borel sets $E,F\subseteq \R^n$ and all $f\in L^2(\R^{n};H)$ with
support in $F$. Here, $\lb a\rb = 1+|a|$ and $d(E,F) = \inf\{|x-y|:\ x\in E, \ y\in F\}$.
The set of such operators is denoted by $OD_{t}(M)$.
\end{definition}

Note that a single operator belongs to \(OD_t(M)\) if and only if it belongs to \(OD_s(M)\) whenever \(s,t>0\). However, the related constant \(C\) will typically not be the same. The scale of the off-diagonal estimates becomes very relevant when we want uniformity in the constants for a family of bounded operators. Thus we say that \((T_z)_{z\in\Sigma}\subseteq L^2(H)\), where \(\Sigma\subseteq\C\), satisfies \emph{off-diagonal estimates of order $M$} if $T_z \in OD_{|z|}(M)$ for all $z\in\Sigma$ with the same constant $C$.


\begin{theorem} \label{tentbdd}
Let $1< p<\infty$, $H$ be a Hilbert space, $X$ be a UMD Banach space, and \(L^{p}(X)\) have type \(\tau\).
Let $(T_t)_{t>0}$ be a uniformly bounded family of operators on $L^2(H)$ satisfying
off-diagonal estimates of order $M$ for some $M > n/\tau$.
Then the operator $T$, defined on $C_c(H)\otimes X$ by
$$ T(g\otimes \xi)(y,t) := T_t (g(\cdot,t))(y) \otimes \xi,$$
extends uniquely to a bounded linear operator on $T^{p,2}(H;X)$.
\end{theorem}

\begin{proof}
Let us consider a function \(f = \sum \limits _i  g_i \otimes \xi_i \in C_c(H)\otimes X\).
We define the sets
\begin{equation*}
\begin{split}
  C_0(x,t) &:= B(x,2t),\\ C_m(x,t) &:= B(x,2^{m+1}t)\setminus B(x,2^m,t),\quad m=1,2,\ldots,
\end{split}\end{equation*}
so that there is a disjoint union \(\bigcup_{m=0}^{\infty} C_m(x,t)=\R^n\). Let \((u_m)_{m=0}^{\infty}\) be the functions
\begin{equation*}
  u_m:x\mapsto\big[(y,t)\mapsto 1_{B(x,t)}(y)T_t\big(1_{C_m(x,t)}f(\cdot,t)\big)(y)\big],
\end{equation*}
where
\begin{equation*}
  T_t\big(1_{C_m(x,t)}f(\cdot,t)\big)(y) := \sum \limits _i T_t (1_{C_m(x,t)}g_i(\cdot,t))(y) \otimes \xi_i.
\end{equation*}
We then have the formal expansion
\(  J(Tf) = \sum_{m=0}^{\infty}u_m \),
and for a fixed \(x\in\R^n\), we separately estimate the \(\gamma(L^2(\dydt;H),X)\)-norms of each \(u_m(x)\).

Fix $\xi^*\in X^*$, and denote by $|\cdot|$ the norm on $H$. Let us also write 
$\lb f(y,t),\xi^* \rb := \sum \limits _i g_i(y,t) \lb \xi_i, \xi^* \rb$.
For $m=0$ we estimate, using the uniform boundedness of the operators $T_t$ on
$L^2(H)$,
\begin{equation*}\begin{split}
  \| u_0(x)^*\xi^*\|_{L^2(\dydt;H)} ^{2}
  &=\int_{\R_+^{n+1}} 1_{B(x,t)}(y) \big|T_t \big(1_{B(x,2t)}\lb f(\cdot,t),\xi\s \rb\big)(y)\big|^2\,\frac{\ud y\ud t}{t^{n+1}}\\
  &\lesssim  \int_{\R_+^{n+1}} 1_{B(x,2t)}(y)|\lb f(y,t) ,\xi\s\rb|^2\,\frac{\ud y\ud t}{t^{n+1}}.
\end{split}\end{equation*}
Hence, by covariance domination (Proposition \ref{prop:CD}),
\begin{equation*}
 \| u_0(x) \|_{\gamma(L^2(\dydt;H),X)}
 \lesssim \| (y,t)\mapsto 1_{B(x,2t)}(y)f(y,t)\|_{\gamma(L^2(\dydt;H),X)},
\end{equation*}
and we conclude that
\begin{equation*}
 \| u_0\|_{L^p(\g(L^2(\dydt;H),X))}
  \lesssim\|f\|_{T^{p,2}_2(H;X)}\lesssim \|f\|_{T^{p,2}(H;X)}.
\end{equation*}

For \(m\geq 1\), the off-diagonal estimates of order \(M\) imply
\begin{equation*}\begin{split}
  \| u_{m}(x)^*\xi^*\|_{L^2(\dydt;H)} ^{2}
  & =\int_{\R_+^{n+1}} 1_{B(x,t)}(y)\big|T_t \big(1_{C_m(x,t)}\lb f(\cdot,t),\xi\s \rb\big)(y)\big|^2
\,\frac{\ud y\ud t}{t^{n+1}}\\
& \leq 2^{-2mM} \int_{\R_+^{n+1}} 1_{B(x,2^{m+1}t)}(y)|\lb f(y,t),\xi\s\rb\big|^2\,\frac{\ud y\ud t}{t^{n+1}}. 
\end{split}\end{equation*}
Hence, by covariance domination,
\begin{equation*}
  \n u_m(x) \n_{\g(L^2(\dydt;H),X)}
  \lesssim 2^{-mM} \n (y,t)\mapsto 1_{B(x,2^{m+1}t)}(y)f(y,t)\n_{\g(L^2(\dydt;H),X)},
\end{equation*} 
and from Theorem \ref{thm:angle} we conclude that
\begin{equation*}
  \n u_m\n_{L^p(\gamma(L^2(\dydt;H),X))} \\
  \lesssim 2^{-mM} \|f\|_{T^{p,2}_{2^{m+1}}(H;X)}
  \lesssim 2^{-mM}\cdot m\cdot 2^{mn/\tau}\|f\|_{T^{p,2}(H;X)}.
\end{equation*}
Keeping in mind that $M>n/\tau$, we may sum over $m$ to see that the formal expansion \(J(Tf)=\sum_{m=0}^{\infty}u_m\) converges absolutely in \(L^p(\gamma(L^2(\dydt;H),X))\), and we obtain the desired result.
\end{proof}

\begin{remark}
The $T^{p,2}(H;X)$-boundedness of the operator $T$ as 
considered above can be seen as a (\(p\) and \(X\) dependent) property of the (parameterised) operator family $(T_t)_{t>0}\subset\calL(L^2(H))$. 
Let us call this property \emph{tent-boundedness}. A simple example of a tent-bounded family 
consists of the translations $T_tf(x)=f(x+t y)$, where $y$ is 
some unit vector. Indeed, these are obviously uniformly bounded 
in $L^2$ (and in $L^p$ as well) 
and satisfy off-diagonal estimates of any order. 
In contrast to this, even when \(X=\C\), it is well known that this family is not $\gamma$-bounded in 
$L^p$ unless $p=2$.
\end{remark}



We next consider operators of the form
\[
 (Tf)_t := \int_{0}^{\infty} T_{t,s}f_s \frac{\ud s}{s},\qquad
  f\in C_c(H) \otimes X,
\]
where \(T_{t,s}\in\calL(L^2(H))\).
This is first done separately for upper and lower diagonal ``kernels'' $T_{t,s}$.

\begin{proposition}\label{prop:upper}
Let $1<p<\infty$, $H$ be a Hilbert space, $X$ be a UMD space, and let $L^p(X)$ have type $\tau$.
Let $(U_{t,s})_{0<t\leq s<\infty}$ be a uniformly bounded family of operators on 
$L^2(H)$ such that $(U_{t,s})_{s\geq t}\in OD_s(M)$ uniformly in $t$ for some $M>n/\tau$. Let further $\alpha>n/2$. Then
\[
  (UF)_t =\int_t^{\infty} \big(\frac{t}{s}\big)^{\alpha}U_{t,s}F_s\frac{\ud s}{s}
\]
extends to a bounded operator on $T^{p,2}(H;X)$.
\end{proposition}

\begin{proof}
Let $F\in C_c(H)\otimes X$ be arbitrary and fixed. 
It suffices to estimate the norm of
the functions $u_k\in L^p(\gamma(L^2(\dydt;H),X))$ defined by
$$
  u_k:x\mapsto \big[(y,t)\mapsto
   1_{B(x,t)}\int_t^{\infty}\big(\frac{t}{s}\big)^{\alpha}U_{t,s}(1_{C_k(x,s)}F_s)(y)\frac{\ud s}{s}\big],
  \quad k=0,1,\ldots,
$$
where $C_0(x,s):=B(x,2s)$, and $C_k(x,s):=B(x,2^{k+1}s)\setminus B(x,2^k s)$ for $k\geq 1$.

Let $x\in\R^n$ be fixed for the moment. To estimate the relevant $\gamma(L^2(\dydt;H),X)$-norm at this point, we wish to use the covariance domination. Hence let $\xi^*\in X^*$, write 
$f_s:=\lb F_s(\cdot),\xi^*\rb \in L^2(H)$ for short, and consider the quantity
\[
   \lb (u_k(x))(y,t),\xi^*\rb
  = 1_{B(x,t)}\int_t^{\infty}\big(\frac{t}{s}\big)^{\alpha}U_{t,s}(1_{C_k(x,s)}f_s)(y)\frac{\ud s}{s} \in H.
\]
Its norm in $L^2(\dydt;H)$ is dominated by
\[\begin{split}
  &\Big(\int_0^{\infty}\Big[\int_t^{\infty}
    \big(\frac{t}{s}\big)^{\alpha}\|1_{B(x,t)}U_{t,s}(1_{C_k(x,s)}f_s)\|_{L^2(H)}\frac{\ud s}{s}\Big]^2
   \frac{\ud t}{t^{n+1}}\Big)^{1/2} \\
  &\leq\Big(\int_0^{\infty}\!\!\Big[\int_t^{\infty}\big(\frac{t}{s}\big)^{2\epsilon}\frac{\ud s}{s}\Big]
    \Big[\int_t^{\infty}\!\!\big(\frac{t}{s}\big)^{2(\alpha-\epsilon)}
    \|1_{B(x,t)}U_{t,s}(1_{C_k(x,s)}f_s)\|_{L^2(H)}^2\frac{\ud s}{s}\Big]
   \frac{\ud t}{t^{n+1}}\Big)^{1/2} \\
  &\lesssim\Big(\int_0^{\infty}
    \int_t^{\infty}\big(\frac{t}{s}\big)^{2(\alpha-\epsilon)}
    \big(2^{-kM}\|1_{B(x,2^{k+1}s)}f_s\|_{L^2(H)}\big)^2\frac{\ud s}{s}
   \frac{\ud t}{t^{n+1}}\Big)^{1/2} \\
  &\eqsim 2^{-kM}\Big(\int_0^{\infty}
       \|1_{B(x,2^{k+1}s)}f_s\|_{L^2(H)}^2
   \frac{\ud s}{s^{n+1}}\Big)^{1/2},
\end{split}\]
where in the last step we exchanged the order of integration and integrated out the $t$ variable; the convergence required that $2(\alpha-\epsilon)>n$, which holds for sufficiently small $\epsilon>0$, since $\alpha>n/2$.

The right-hand side of our computation is $2^{-kM}$ times the $L^2(\dydt;H)$-norm of $1_{B(x,2^{k+1}s)}\lb F_s(y),\xi^*\rb$, 
so that covariance domination gives us
\[
  \|u_k(x)\|_{\gamma(L^2(\dydt;H),X)}\lesssim 2^{-kN}\|(J_{2^{k+1}}F)(x)\|_{\gamma(L^2(\dydt;H),X)}.
\]
Taking $L^p$-norms and using Theorem~\ref{thm:angle} yields
\[
  \|u_k\|_{L^p(\gamma(L^2(\dydt;H),X))}
  \lesssim 2^{-kM}\|F\|_{T^{p,2}_{2^{k+1}}(H;X)}
  \lesssim 2^{-kM}(1+k)2^{kn/\tau}\|F\|_{T^{p,2}(H;X)}.
\]
Recalling that $M>n/\tau$, we find that the formal expansion $J(UF)=\sum_{k=0}^{\infty}u_k$ converges absolutely in $L^p(\gamma(L^2(\dydt;H),X))$, and we obtain the desired estimate $\|UF\|_{T^{p,2}(X)}\lesssim\|F\|_{T^{p,2}(X)}$. 
\end{proof}

\begin{proposition}\label{prop:lower}
Let $1<p<\infty$, $H$ be a Hilbert space, $X$ be a UMD space, and let $L^p(X)$ have type $\tau$.
Let $(L_{t,s})_{0<s\leq t<\infty}$ be a uniformly bounded family of operators on 
$L^2(H)$ 
such that $(L_{t,s})_{t \geq s}\in OD_t(N)$ uniformly in $s$  for some $N>n/\tau$. Let further $\beta>n(1/\tau-1/2)$. Then
\[
  (LF)_t =\int_0^t \big(\frac{s}{t}\big)^{\beta}L_{t,s}F_s\frac{\ud s}{s}
\]
extends to a bounded operator on $T^{p,2}(H;X)$.
\end{proposition}

\begin{proof}
The proof follows a similar approach as the previous one. This time, we expand $J(LF)$ in a double series $\sum_{k,m=0}^{\infty}v_{k,m}$, where
\[
  v_{k,m}:x\mapsto\big[(y,t)\mapsto
  \int_{2^{-(m+1)}t}^{2^{-m}t}\big(\frac{s}{t}\big)^{\beta}1_{B(x,t)}(y)
    L_{t,s}(1_{C_k(x,t)}F_s)(y)\frac{\ud s}{s}\big].
\]
Again, we wish to estimate the $\gamma(L^2(\dydt;H),X)$-norm of $v_{k,m}(x)$ by covariance domination, for which purpose we take $\xi^*\in X^*$, write $f_s:=\lb F_s(\cdot),\xi^*\rb$, and compute
\[\begin{split}
  &\|\lb v_{k,m}(x),\xi^*\rb\|_{L^2(\dydt;H)} \\
  &\leq\Big(\int_0^{\infty}\Big[\int_{2^{-(m+1)}t}^{2^{-m}t}2^{-m\beta}
    \|1_{B(x,t)}L_{t,s}(1_{C_k(x,t)}F_s)\|_{L^2(H)}\frac{\ud s}{s}\Big]^2\frac{\ud t}{t^{n+1}}\Big)^{1/2} \\
  &\lesssim 2^{-m\beta}\Big(\int_0^{\infty}\int_{2^{-(m+1)}t}^{2^{-m}t}
     \big(2^{-kN}\|1_{B(x,2^{k+1}t)}F_s\|_{L^2(H)}\big)^2\frac{\ud s}{s}\frac{\ud t}{t^{n+1}}\Big)^{1/2} \\
  &\lesssim 2^{-m(\beta+n/2)}2^{-kN}\Big(\int_0^{\infty}
     \|1_{B(x,2^{k+m+2}s)}F_s\|_{L^2(H)}^2
\frac{\ud s}{s^{n+1}}\Big)^{1/2}.
\end{split}\]
This is $2^{-m(\beta+n/2)}2^{-kN}$ times the $L^2(\dydt;H)$-norm of $1_{B(x,2^{k+m+2}s)}(y)\lb F_s(y),\xi^*\rb$; 
hence by covariance domination
\[
  \|v_{k,m}(x)\|_{\gamma(L^2(\dydt;H),X)}
  \lesssim 2^{-m(\beta+n/2)}2^{-kN}\|(J_{2^{k+m+2}}F)(x)\|_{\gamma(L^2(\dydt;H),X)}.
\]
Taking $L^p$-norms and using Theorem~\ref{thm:angle} we get
\[\begin{split}
  &\|v_{k,m}\|_{L^p(\gamma(L^2(\dydt;H),X))}
  \lesssim 2^{-m(\beta+n/2)}2^{-kN}\|F\|_{T^{p,2}_{2^{k+m+2}}(H;X)} \\
  &\lesssim 2^{-m(\beta+n/2)}2^{-kN}(1+k+m)2^{(k+m)n/\tau}\|F\|_{T^{p,2}(H;X)},
\end{split}\]
and we can sum up the series over $k$ and $m$ since $\beta+n/2>n/\tau$ and $N>n/\tau$.
\end{proof}

Combining the previous two propositions with a duality argument, we finally obtain:

\begin{theorem}\label{thm:Schur}
Let $1<p<\infty$, $H $ be a Hilbert space, $X$ be a UMD space, and let $L^p(X)$ have type $\tau$ and cotype $\gamma$.
Let $(T_{t,s})_{0<t,s<\infty}$ be a uniformly bounded family of operators on 
$L^2(H)$ such that:
\begin{enumerate}
  \item[$(i)$] $(T_{t,s})_{s>t}\in OD_s(M)$ uniformly in $t$,
  \item[$(ii)$] $(T_{t,s})_{t>s}\in OD_t(N)$ uniformly in $s$.
  \end{enumerate}
Then
$$
  (TF)_t =\int_0^{\infty}
   \min\big\{\big(\frac{t}{s}\big)^{\alpha},\big(\frac{s}{t}\big)^{\beta}\big\}
     T_{t,s}F_s\frac{\ud s}{s}
$$
extends to a bounded operator on $T^{p,2}(H;X)$ if at least one of the following four conditions is satisfied:
\begin{enumerate}
  \item[$(a)$] $M>n/\tau$, $\alpha>n/2$, $N>n/\tau$, and $\beta>n(1/\tau-1/2)$, 
  \item[$(b)$] $M>n/\tau$, $\alpha>n/2$, $N>n(1-1/\gamma)$, and $\beta>n/2$, 
  \item[$(c)$] $M>n(1-1/\gamma)$, $\alpha>n(1/2-1/\gamma)$, $N>n/\tau$, and $\beta>n(1/\tau-1/2)$, 
  \item[$(d)$] $M>n(1-1/\gamma)$, $\alpha>n(1/2-1/\gamma)$, $N>n(1-1/\gamma)$, and $\beta>n/2$.
\end{enumerate}
\end{theorem}

\begin{proof}
We split $T$ into a sum $U+L$ of upper and lower triangular parts as considered in the previous two propositions. Part $(a)$ is an immediate consequence, since the conditions on $M$ and $\alpha$ guarantee the boundedness of $U$ and those on $N$ and $\beta$ that of $L$.

For part $(b)$, the boundedness of $U$ follows as before. As for $L$, we observe that its (formal) adjoint on $T^{p',2}(H;X^*)$ is the upper triangular operator
\[
  (L^* G)_t=\int_t^{\infty}\big(\frac{t}{s}\big)^{\beta}T_{s,t}^*G_s\frac{\ud s}{s},
\] 
where $T_{s,t}^*\in OD_s(N)$ and $L^{p'}(X^*)=(L^p(X))^*$ has type $\gamma'=\gamma/(\gamma-1)$. We know that this operator is bounded on $T^{p',2}(H;X^*)$ under the conditions that $N>n/\gamma'=n(1-1/\gamma)$ and $\beta>n/2$.

Parts $(c)$ and $(d)$ are proved similarly by considering $U^*$ and $L$, and $U^*$ and $L^*$, respectively.
\end{proof}

The most important case for us is when \(N=M\), and we record this as a corollary for later reference. In this situation, the condition $(b)$ of Theorem~\ref{thm:Schur} becomes redundant, since it is always contained in condition $(a)$.

\begin{corollary}\label{cor:Schur}
Let $1<p<\infty$, $H$ be a Hilbert space, $X$ be a UMD space, and let $L^p(X)$ have type~$\tau$ and cotype~$\gamma$.
Let $(T_{t,s})_{0<t,s<\infty}$ be a uniformly bounded family of operators on 
$L^2(H)$ such that \(T_{t,s}\in OD_{\max\{t,s\}}(M)\) uniformly in \(t\) and \(s\).
Then
\begin{equation}\label{eq:opT2}
  (TF)_t =\int_0^{\infty}
   \min\big\{\big(\frac{t}{s}\big)^{\alpha},\big(\frac{s}{t}\big)^{\beta}\big\}
     T_{t,s}F_s\frac{\ud s}{s}
\end{equation}
extends to a bounded operator on $T^{p,2}(H;X)$ if at least one of the following three conditions is satisfied:
\begin{enumerate}
  \item[$(a)$] $M>n/\tau$, $\alpha>n/2$, and $\beta>n(1/\tau-1/2)$, 
  \item[$(c)$] $M>n\cdot\max\{1/\tau,1-1/\gamma\}$, $\alpha>n(1/2-1/\gamma)$, and $\beta>n(1/\tau-1/2)$, 
  \item[$(d)$] $M>n(1-1/\gamma)$, $\alpha>n(1/2-1/\gamma)$, and $\beta>n/2$.
\end{enumerate}
\end{corollary}

\begin{remark}
If $X=\C$ (or more generally a Hilbert space), then one can take 
$\tau=\min(2,p)$ and $\gamma=\max(2,p)$ in Corollary~\ref{cor:Schur}. 
For $p\in[2,\infty)$ (so that $\tau=2$), 
part $(a)$ provides the following sufficient condition for the $T^{p,2}$-boundedness of~\eqref{eq:opT2}: 
$M,\alpha>n/2$, and $\beta>0$. 
For $p\in(1,2]$ (so that $\gamma=2$), part $(d)$ in turn gives 
$M,\beta>n/2$, and $\alpha>0$. This recovers the corresponding result in 
\cite{AMR} in the Euclidean case for $p\in(1,\infty)$. 
Note that in \cite{AMR} the end-points $p\in\{1,\infty\}$ are 
also considered; in fact, the proof for $p\in(1,2)$ goes 
via interpolating between estimates available in the 
atomic space $T^{1,2}$ and the Hilbert space $T^{2,2}$.  
See also \cite{ADM}, where a weak type $(1,1)$ estimate is obtained.
\end{remark}

\section{Bisectorial operators and functional calculus}\label{sect:calculus}

In this section we collect some generalities concerning bisectorial operators and their 
\(H^{\infty}\)-calculus.
We denote by \(S_{\theta}\) the (open) bisector of angle \(\theta\), i.e.
$S_\theta = S_\theta^+\cup S_\theta^-$
with $S_\theta^+ = \{z\in \C\setminus\{0\}: \ |\arg(z)|<\theta\}$
and $S_\theta^- =-S_\theta^+$.
We denote by \(\Gamma_\theta\) the boundary of \(S_{\theta}\),
which is parameterised by arc-length and oriented anticlockwise around \(S_{\theta}\).

A closed, densely defined, linear operator $A$ acting in a Banach space \(Y\) is called {\em bisectorial} (of {\em angle} $\omega$,
where $0<\omega<\frac12\pi$)
if the spectrum of $A$ is contained in $\overline{S_\omega}$ and 
for all $\omega<\theta<\frac12\pi$ there exists a constant $C_\theta$
such that for all nonzero $z\in \C\setminus S_\theta$
\begin{equation*}
  (I+zA)^{-1} \le C_\theta \frac{|z|}{d(z,S_\theta)}.
\end{equation*}

For \(\alpha,\beta >0\) we set 
\begin{equation*}\begin{split}
  \Psi_{\alpha}(S_{\theta})
  & = \big\{f \in H^{\infty}(S_{\theta}): \ \exists C \ \ |f(z)| \leq C\min(|z|^{\alpha},1) 
      \hbox{ for all } z\in S_\theta\big\}, \\
  \Psi^{\beta}(S_{\theta})
  & = \big\{f \in H^{\infty}(S_{\theta}): \ \exists C \ \ |f(z)| \leq C\min(1,|z|^{-\beta}) 
      \hbox{ for all } z\in S_\theta\big\}, \\
  \Psi_{\alpha}^\beta(S_{\theta}) & = \big\{f \in H^{\infty}(S_{\theta}): \ \exists C \ \ |f(z)| \leq C\min(|z|^{\alpha},|z|^{-\beta})  \hbox{ for all } z\in S_\theta\big\}
\end{split}\end{equation*}
and 
\(\Psi(S_{\theta}) = \bigcup  _{\alpha,\beta>0}
\Psi_{\alpha}^{\beta}(S_{\theta}).\)

Let \(\omega<\theta<\frac12\pi\) be fixed. 
For \(\psi\in \Psi(S_\theta)\), we define
\begin{equation*}
 \psi(A) =  \frac{1}{2\pi i} \int_{\Gamma_\theta}\psi(z)(z-A)^{-1}\ud z.
\end{equation*}
The resolvent bounds for $A$ imply that this integral converges absolutely
in $\calL(Y)$. If one has, in addition, the quantitative estimate \[\|\psi(A)\|_{\calL(Y)}\lesssim\|\psi\|_{\infty},\] then \(A\) is said to have 
\(H^{\infty}(S_{\theta})\)-calculus 
on \(Y\).

\begin{lemma}\label{lem:justification}
Let $A$ be bisectorial of angle $\omega$ and let $\theta>\omega$.
\ben
\item[\rm(1)]
For $\phi_1,\phi_2\in\Psi(S_{\theta})$ we have $\phi_1(A)\phi_2(A)=(\phi_1\cdot\phi_2)(A)$; this is also true if 
$\phi_2\in H^{\infty}(S_{\theta})$ is a rational function, in which case $\phi_2(A)$ is defined in the usual way by using the resolvents of $A$.
\item[\rm(2)]
For all
$\psi_1\in \Psi(S_\theta)$, $\psi_2\in H^\infty(S_\theta)$, 
$\psi_3\in \Psi(S_\theta)$ we have
\begin{equation*}
  \psi_1(A)(\psi_2\psi_3)(A) = (\psi_1\psi_2)(A)\psi_3(A).
\end{equation*}
\een
\end{lemma}

\begin{proof}
The first claim is the well-known homomorphism property, which in both cases can be proved by writing out the definition of $\phi_1(A)\phi_2(A)$, performing a partial fraction expansion, and using Cauchy's theorem. The second claim follows from the homomorphism property for $\psi_2\in \Psi(S_\theta)$, and the general case can be obtained from this by approximation
(cf. \cite[Theorem 9.2(i)]{KuWe}).
\end{proof} 

\begin{lemma}\label{lem:ranges}
Let $A$ be bisectorial of angle $\omega$ and let  $\theta>\omega$. Then,
\[\overline{\ran(A)}=\overline{\ran(A)\cap\dom(A)}
=\overline{\ran(A(I+A)^{-2})}=\overline{\bigcup_{\psi\in\Psi(S_\theta)}
\ran(\psi(A))}.\]
\end{lemma}

\begin{proof}
If \(f=\psi(A)g\in\ran(\psi(A))\), let \(f_{\varepsilon}:=A(\varepsilon+A)^{-1}f\in\ran(A)\). Then
\begin{equation*}
  f-f_{\varepsilon}=\varepsilon(\varepsilon+A)^{-1}\psi(A)g
  =\frac{1}{2\pi i}\int_{\Gamma}
   \frac{\varepsilon}{\varepsilon+z}\psi(z)(z-A)^{-1}g\ud z.
\end{equation*}
The integrand is bounded by \(\abs{\psi(z)z^{-1}}\in L^1(\Gamma,\abs{\ud z})\) and tends pointwise to zero as \(\varepsilon\to 0\). Hence \(f_{\varepsilon}\to f\) by dominated convergence.

Next we observe that \(f^{\varepsilon}=(I+\varepsilon A)^{-1}f\to f\) as \(\varepsilon\to 0\). Indeed, if \(f\in\dom(A)\), then \(f-f^{\varepsilon}=\varepsilon\cdot(I+\varepsilon A)^{-1}Af\) has norm at most \(C\varepsilon\), since the second factor stays uniformly bounded. Since the operators \((I+\varepsilon A)^{-1}\) are uniformly bounded and \(\dom(A)\) is dense, the convergence remains true for all \(f\). If now \(f\in\ran(A)\), then \(f^{\varepsilon}\in\ran(A)\cap\dom(A)\).

To complete the chain, let \(f\in\ran(A)\cap\dom(A)\). Then for some \(g\in\dom(A^2)\) we have  \(f=Ag=A(I+A)^{-2}(I+A)^2g=\psi(A)h\), where \(\psi(z)=z/(1+z)^2\in\Psi\) and \(h=(I+A)^2g\). This completes the proof.
\end{proof}

We say that \(\psi\in \Psi_{\a}^{\b}(S_\theta)\) is degenerate if (at least) one of the restrictions \(\psi|_{S^{\pm}_{\theta}}\) vanishes identically; otherwise it is called non-degenerate.
The following two lemmas go back to Calder\'on, cf. \cite[Section
IV.6.19]{Stein}. For the convenience of the reader we include simple proofs.

\begin{lemma}[Calder\'on's reproducing formula, I]\label{lem:newPsi} 
Let \(\psi\in \Psi_{\a}^{\b}(S_\theta)\) be non-degen\-erate. If $\a'\geq \a$ and $\b'\geq \b$,
there exists \(\widetilde \psi\in  \Psi_{\a'}^{\b'}(S_\theta)\)
such that 
\begin{equation}\label{eq:Calderon}
  \int_0^\infty \psi(tz)\widetilde\psi(tz)\frac{\ud t}{t} = 1, \quad z\in S_\theta.
\end{equation}
\end{lemma}

\begin{proof}
Let $\overline\psi(z) := \overline{\psi(\overline z)}$. Let $m\ge \max(\a'-\a,
\b'-\b)$ and denote
\begin{equation*}
  c_{\pm}:=\int_0^\infty \frac{(\pm t)^m}{(1+t^2)^{m}} \psi(\pm t)\overline\psi(\pm t)\frac{\ud t}{t}.
\end{equation*}
By non-degeneracy, \(c_{\pm}> 0\).
Hence the function $\widetilde\psi(z) =
c_{\pm}^{-1}z^m(1+z^2)^{-m}\overline{\psi(\overline{z})}$ for \(z\in S^{\pm}_{\theta}\) has the desired properties.
\end{proof}

\begin{lemma}[Calder\'on's reproducing formula, II]\label{lem:reproducing}
Let \(\psi,\widetilde\psi\in\Psi(S_{\theta})\) satisfy \eqref{eq:Calderon}. Then
\begin{equation*}
  \int_0^{\infty}\psi(tA)\widetilde\psi(tA)f\frac{\ud t}{t}=f,\qquad
  f\in\overline{\ran(A)},
\end{equation*}
where the left side is defined as an indefinite Riemann integral in \(L^2\).
\end{lemma}

\begin{proof}
Let first \(f=\phi(A)g\) for some \(\phi\in\Psi(S_{\theta})\). Then
\begin{equation*}\begin{split}
  \int_0^{\infty}\psi(tA)\widetilde\psi(tA)f\frac{\ud t}{t}
  &=\int_0^{\infty}(\psi(t\cdot)\widetilde\psi(t\cdot)\phi(\cdot))(A)g\frac{\ud t}{t} \\
  &=\int_0^{\infty}  \frac{1}{2\pi i}\int_{\Gamma_{\theta'}}
   \psi(tz)\widetilde\psi(tz)\phi(z)(z-A)^{-1}g\ud z\frac{\ud t}{t} \\
  &=\frac{1}{2\pi i}\int_{\Gamma_{\theta'}}
   \int_0^{\infty}\psi(tz)\widetilde\psi(tz)\frac{\ud t}{t}\phi(z)(z-A)^{-1}g\ud z \\
  &=\frac{1}{2\pi i}\int_{\Gamma_{\theta'}}\phi(z)(z-A)^{-1}g\ud z
   =\phi(A)g=f
\end{split}\end{equation*}
by Lemma~\ref{lem:justification}, absolute convergence and Fubini's theorem. To conclude, we recall from Lemma~\ref{lem:ranges} that functions as above are dense in \(\overline{\ran(A)}\), and notice that \(\int_a^b\psi(sz)\widetilde\psi(sz)\ud s/s\) are uniformly in 
\(H^{\infty}(S_{\theta})\) so that the corresponding operators obtained by the formal substitution \(z:=A\) are uniformly bounded by the functional calculus. From this the convergence of the indefinite Riemann integral to the asserted limit follows easily.
\end{proof}

\section{Hardy spaces associated with bisectorial operators}
\label{sect:hardy}

We now move on to more specific spaces and operators. Throughout this section, we let the following assumptions be satisfied:

\begin{assumption}\label{as:space}
The Banach space \(X\) is UMD and \(1<p<\infty\). Two numbers \(\tau\in[1,2]\) and \(\gamma\in[2,\infty]\) are fixed in such a way that \(L^p(X)\) has type~\(\tau\) and cotype~\(\gamma\). 
\end{assumption}

\begin{assumption}\label{as:operator}
$H$ is a Hilbert space, and the operator $A$ in \(L^2(H)\) 
 is bisectorial of angle \(\omega \in (0,\pi/2)\). For \(\omega<\theta'<\theta<\pi/2\), it also has an 
 \(H^{\infty}(S_{\theta})\)-calculus
 on \(L^2(H)\), 
and the family \(((I+\zeta A)^{-1})_{\zeta\in \C\setminus S_{\theta}}\)
satisfies off-diagonal estimates of order $M$, where \(M>n\cdot\min\{1/\tau,1-1/\gamma\}\).
\end{assumption}






With only the above assumptions at hand, it may well happen that \(A\)  fails to be bisectorial 
even for $H=\C$, 
and in particular to have an \(H^{\infty}\)-calculus, 
in \(L^p\) for some values of \(p\not=2\). The tensor extension \(A\otimes I_X\) may already fail these properties in \(L^2(X)\). To study problems involving operators \(f(A)\) in such spaces, we are 
thus led to define an appropriate scale of Hardy spaces associated with \(A\). When \(A\) is the Hodge--Dirac operator or the Hodge--de Rham Laplacian on a complete Riemannian manifold, 
this has been done in \cite{AMR}.
We build on the ideas of this paper.

\begin{lemma}
\label{lem:od1}
For \(\omega<\theta<\pi/2\) 
and $\varepsilon>0$, 
let \(g\in H^{\infty}(S_{\theta})\), 
and let \(\psi \in \Psi_{M+\varepsilon}^{\varepsilon}(S_{\theta})\).
Then \(\{(g\cdot\psi(t\cdot))(A)\}_{t>0}\) satisfies off-diagonal estimates of order \(M\), and the off-diagonal constant has an upper bound which depends linearly on \(\|g\|_{\infty}\).
\end{lemma}

\begin{proof}
Let us denote by \(\delta:=d(E,F)\) the `distance' 
of two Borel sets \(E\) and \(F\)
as defined previously. 
Then, using the fact that \((I-z^{-1}\,A)^{-1}\in OD_{1/|z|}(M)\) uniformly in \(z\in S_{\theta}\),
\begin{equation*}\begin{split}
  \|1_E &(g\cdot\psi(t\cdot))(A)1_F f\| \\
  &=\Big\|\frac{1}{2\pi i}\int_{\Gamma_{\theta'}}
     g(z)\psi(tz)1_E\big(I-\frac{1}{z}A\big)^{-1}1_F f\frac{\ud z}{z}\Big\| \\
  &\lesssim\int_{\Gamma_{\theta'}}\min\big\{(t|z|)^{M+\varepsilon},(t|z|)^{-\varepsilon}\big\}
     (\delta|z|)^{-M}\|f\|\frac{|\ud z|}{|z|} \\
  &\lesssim\int_0^{1/t}t^{M+\varepsilon}r^{M+\varepsilon}\cdot\delta^{-M}r^{-M}\|f\|\frac{\ud r}{r}
    +\int_{1/t}^{\infty}t^{-\varepsilon}r^{-\varepsilon}\cdot\delta^{-M}r^{-M}\|f\|\frac{\ud r}{r} \\
  &\eqsim t^M\delta^{-M}\|f\|,
\end{split}\end{equation*}
and this proves the claim.
\end{proof}

\begin{lemma}
\label{lem:od2}
Let \(\alpha,\beta,\varepsilon>0\), and
\begin{equation*}
 \psi\in \Psi_{\max\{M-\beta,\alpha\}+\varepsilon}^{\beta+\varepsilon}(S_{\theta}),\quad
 \widetilde{\psi} \in \Psi_{\max\{M-\alpha,\beta\}+\varepsilon}^{\alpha+\varepsilon}(S_{\theta}),
 \quad\phi \in \C1\oplus \Psi(S_{\theta}).
\end{equation*} 
Then 
\begin{equation*}
 \psi(tA) \phi(A) \widetilde{\psi}(sA) =
 \min\Big\{\big(\frac{t}{s}\big)^{\alpha},\big(\frac{s}{t}\big)^{\beta}\Big\} S_{t,s},
\end{equation*} 
where \((S_{t,s})_{t,s>0}\) is a uniformly bounded family of operators acting on
$L^2(H)$ such that \(S_{t,s}\in OD_{\max\{t,s\}}(M)\), uniformly in \(t\) and \(s\).
\end{lemma} 

\begin{proof}
We have
\begin{equation*}
  \psi(tA) \phi(A) \widetilde{\psi}(sA)
  = (t/s)^{\alpha}\psi_0(tA)\phi(A)\widetilde\psi_0(sA)
  = (s/t)^{\beta}\psi_1(tA)\phi(A)\widetilde\psi_1(sA),
\end{equation*}
where
\begin{alignat*}{2}
    &\psi_0(z) :=z^{-\alpha}\psi(z)\in\Psi^{\alpha+\beta+\varepsilon}_{\varepsilon},\qquad
    &\widetilde\psi_0(z) :=z^{\alpha}\widetilde\psi(z)\in\Psi^{\varepsilon}_{M+\varepsilon}, \\
    &\psi_1(z) :=z^{\beta}\psi(z)\in\Psi^{\varepsilon}_{M+\varepsilon},\qquad
    &\widetilde\psi_1(z) :=z^{-\alpha}\widetilde\psi(z)\in\Psi^{\alpha+\beta+\varepsilon}_{\varepsilon}.
\end{alignat*}
The case \(s\geq t\) of the claim follows from  Lemma~\ref{lem:od1} (with \(s\) in playing the role of \(t\) in that Lemma) with \(g(z)=\psi_0(tz)\phi(z)\) and \(\widetilde\psi_0\) in place of \(\psi\), while for the other case we take \(g(z)=\phi(z)\widetilde\psi_1(sz)\) and \(\psi_1\) in place of \(\psi\).
\end{proof}

\begin{proposition}\label{prop:Schur}
Let \(\psi,\tilde\psi\in\Psi(S_{\theta})\) and \(\phi\in\C1\oplus\Psi(S_{\theta})\). Then
\begin{equation*}
  (TF)_t=\int_0^{\infty}\psi(tA)\phi(A)\psi(sA)F_s\frac{\ud s}{s}
\end{equation*}
extends to a bounded operator on \(T^{p,2}(H;X)\) if at least one of the following conditions is satisfied:
\begin{enumerate}
  \item[$(a)$] \(M>n/\tau\), \(\psi\in\Psi^{n(1/\tau-1/2)+\varepsilon}_{n/2+\varepsilon}\), and
     \(\tilde\psi\in\Psi^{n/2+\varepsilon}_{n(1/\tau-1/2)+\varepsilon}\),
  \item[$(c)$] \(M>\max\{n/\tau,n(1-1/\gamma)\}\),
  \(\psi\in\Psi^{n(1/\tau-1/2)+\varepsilon}_{n/2+n\max\{1/\gamma'-1/\tau,0\} +\varepsilon}\), \\ and
    \(\tilde\psi\in\Psi^{n(1/2-1/\gamma)+\varepsilon}_{n/2+n\max\{1/\tau-1/\gamma',0\}+\varepsilon}\),
  \item[$(d)$] \(M>n(1-1/\gamma)\), \(\psi\in\Psi^{n/2+\varepsilon}_{n(1/2-1/\gamma)+\varepsilon}\), and
     \(\tilde\psi\in\Psi^{n(1/2-1/\gamma)+\varepsilon}_{n/2+\varepsilon}\),
\end{enumerate}
where \(\varepsilon>0\) is arbitrary.
\end{proposition}

\begin{proof}
This is directly, if slightly tediously, verified as a corollary of Lemma~\ref{lem:od2} and Corollary~\ref{cor:Schur}, so that the different conditions of Proposition~\ref{prop:Schur} correspond to those of Corollary~\ref{cor:Schur}.
\end{proof}

\begin{definition}\label{def:decay}
We say that a pair of functions \((\psi,\tilde\psi)\in\Psi(S_{\theta})\times\Psi(S_{\theta})\) has \emph{sufficient decay} if they verify at least one of the conditions $(a)$, $(c)$, or $(d)$ of Proposition~\ref{prop:Schur}.
\end{definition}

\begin{remark}\label{rem:decay}
$(i)$ Note that the notion of sufficient decay as defined above assumes that the parameters appearing in Assumptions~\ref{as:space} and \ref{as:operator} have been fixed. Also observe that if the parameters are such that for instance \(n(1-1/\gamma)<M\leq n/\tau\), then only the condition $(d)$ above is applicable.

$(ii)$ If \((\psi,0)\in\Psi(S_{\theta})\times \Psi(S_{\theta})\) has sufficient decay, by Calder\'on's reproducing formula there exists a \(\tilde\psi\in\Psi(S_{\theta})\) which satisfies \eqref{eq:Calderon} and decays as rapidly as desired; in particular, we may arrange so that the pair \((\psi,\tilde\psi)\) also has sufficient decay. A similar remark applies if we start from a \(\tilde\psi\in \Psi(S_{\theta})\) such that \((0,\tilde\psi)\) has sufficient decay.
\end{remark}


For $f = \sum \limits _i g_i \otimes \xi_i\in L^2\otimes X$ and $\psi\in \Psi(S_\theta)$ 
we shall write $$(Q_{\psi}f)(y,t) := \sum \limits _i \psi(tA)g_i(y) \otimes \xi_i:= \psi(tA)f(y).$$

\begin{definition}
For \(1 \leq p < \infty\) and a non-degenerate $\psi\in \Psi(S_\theta)$, the {\em Hardy space
$H^{p} _{A, \psi} (X)$ associated with \(A\) and \(\psi\)} is the completion of
the space 
\begin{equation*}
  \{f\in \overline{\ran(A)}\otimes X\subseteq L^2(H)\otimes X: \ Q_{\psi}f\in T^{p,2}(X)\}
\end{equation*}
with respect to the norm
\begin{equation*}
 \n f\n_{H^{p} _{A, \psi} (X)}:= \n Q_{\psi}f
 \n_{ T^{p,2}(H;X)}.
\end{equation*}
\end{definition}

It is clear that \(\n\cdot\n_{H^{p} _{A, \psi} (X)}\) is a seminorm
on \(\overline{\ran(A)}\otimes X\); that it is actually a norm will be seen shortly.


By definition, the operator
\begin{equation*}
  (Q_{\psi} f)(\cdot,t):= \psi(tA)f
\end{equation*}  
embeds the Hardy space \(H^p_{A,\psi}(H;X)\) isometrically into the tent space \(T^{p,2}(H;X)\). Of importance will also be another operator acting to the opposite direction.
For \(\widetilde\psi\in\Psi(S_{\theta})\),
we define \(S_{\widetilde{\psi}} f\in L^2(\R^n;H)\otimes X\) by 
\begin{equation}\label{eq:Spsi}
  S_{\widetilde\psi} F:= \int_0^\infty \widetilde\psi(sA)F(s,\cdot)\frac{\ud s}{s}
\end{equation}
for those functions \(F\in L^1_{\mathrm{loc}}(\R_+;L^2(\R^n;H))\otimes X\) for which the integral exists as a limit in \(L^2(H)\) of the finite integrals \(\int_a^b\), where \(a\to 0\) and \(b\to\infty\).

By Calder\'on's reproducing formula, for a given \(\psi\in\Psi(S_{\theta})\), there exists a \(\tilde\psi\in\Psi(S_{\theta})\) such that the defining formula~\eqref{eq:Spsi} makes sense for all \(F\in Q_{\psi}(\overline{\ran(A)}\otimes X)\), and we have
\begin{equation}\label{eq:SQ}
  S_{\tilde\psi}Q_{\psi}f=f,\qquad f\in \overline{\ran(A)}\otimes X.
\end{equation}
Hence, if \(\|f\|_{H^p_{A,\psi}(H;X)}=0\) for some \(f\in\overline{\ran(A)}\otimes X\), this means by definition that \(Q_{\psi}f=0\), and the identity \eqref{eq:SQ} yields immediately \(f=0\). Thus \(\|\cdot\|_{H^p_{A,\psi}(H;X)}=0\) is indeed a norm.

\begin{proposition}\label{prop:S_psi}
Let \((\psi,\widetilde\psi)\in\Psi(S_{\theta})\times \Psi(S_{\theta})\) 
be a pair with sufficient decay.
If \(f\in T^{p,2}(H;X)\) is such that the defining formula \eqref{eq:Spsi} is valid, then  \(S_{\widetilde\psi} f \in H^{p} _{A, \psi} (H;X)\), and the mapping \(f\mapsto S_{\widetilde\psi} f\) extends uniquely 
to a bounded operator from \(T^{p,2}(H;X)\) to \(H^{p}_{A,\psi} (H;X)\).
\end{proposition}

\begin{proof}
Write $g:=S_{\widetilde\psi} f$. First we check that
\(g\in \overline{\ran(A)}\otimes X\): this is clear from the defining formula, since \(\psi(sA)f(\cdot,s)\in\overline{\ran(A)}\) for each \(s>0\) by Lemma~\ref{lem:ranges}, and Bochner integration in the Banach space \(L^2(H)\) preserves the closed subspace \(\overline{\ran(A)}\).

By Proposition~\ref{prop:Schur},
\begin{equation*}
  (y,t)\mapsto \psi(tA)g (y) =
  \int_0^\infty \psi(tA)\widetilde\psi(sA)f(y,s)\frac{\ud s}{s}
\end{equation*}
defines an element $\psi(\cdot A)g$ of $T^{p,2}(H;X)$ and we have
\begin{equation*}
 \n S_{\widetilde\psi} f\n_{H^{p} _{A, \psi} (H;X)} = \n \psi(\cdot A)g\n_{T^{p,2}(H;X)}
  \lesssim \n f\n_{T^{p,2}(H;X)}.
\end{equation*}
The subspace of \(T^{p,2}(H;X)\) where the defining formula~\eqref{eq:Spsi} is valid contains e.g.~\(C_c(H)\otimes X\) and is therefore dense in \(T^{p,2}(H;X)\). Hence the mapping 
\(S_{\widetilde\psi}\) has a unique extension to a bounded operator from 
$T^{p,2}(H;X)$ to $H^p_{A,\psi}(H;X)$.
\end{proof}

Next we show that \(H^{p} _{A, \psi} (H;X)\) is independent of \(\psi \in
\Psi(S_{\theta})\), provided $(\psi,0)$ has sufficient decay.
A typical function with this property is 
\begin{equation*}
  \psi(z)=(\sqrt{z^{2}})^{n(\frac{1}{2}-\frac{1}{\gamma})+1}e^{-\sqrt{z^{2}}},
\end{equation*}
where \(\gamma\) denotes the cotype of $L^{p}(X)$. This gives the classical definition by the Poisson kernel when $X=\C$ and $1<p\leq 2$, taking \(\gamma=2\).

\begin{theorem}
\label{thm:hardy}
Let \(\psi, \underline{\psi} \in \Psi(S_{\theta})\) be two functions such that \((\psi,0)\) and \((\underline\psi,0)\) have sufficient decay. Then:
\begin{enumerate}
\item[\rm(i)]
\(H^{p} _{A,\psi} (H;X) = H^{p}_{A, \underline{\psi}}(H;X) =: H_A^p(H;X)\).
\item[\rm(ii)]
\(A\) has an \(H^{\infty}\)-functional calculus on \(H^{p} _{A}(H;X)\).
\end{enumerate}
\end{theorem}

\begin{proof}
Let \(\phi \in\C1\oplus \Psi(S_{\theta})\) be arbitrary and fixed.
Let \(f\in \overline{\ran(A)}\otimes X\). By Calder\'on's reproducing formula, there exists
\(\widetilde{\psi} \in  \Psi(S_{\theta})\) (with any prescribed decay) such that
\begin{equation*}
  \psi(tA)\phi(A)f 
  =\int_{0}^{\infty} \psi(tA)\phi(A)\widetilde{\psi}(sA)\underline{\psi}(sA)f \frac{\ud s}{s}.
\end{equation*}
Thus \[\|\phi(A)f\|_{H^{p}_{A,\psi}(H;X)}=\| TQ_{\underline{\psi}}f\|
_{T^{p,2}(H;X)},\] where $T$ is the operator on $T^{p,2}(H;X)$ given by
\begin{equation*}
  TF(y,t) = 
  \int _{0} ^{\infty} \psi(tA)\phi(A)\widetilde{\psi}(sA)F(y,s)\frac{\ud s}{s}.
\end{equation*}
From Proposition~\ref{prop:Schur} we deduce that
\begin{equation*}
  \|\phi(A)f\|_{H^{p}_{A,\psi}(H;X)}
 \lesssim 
  \|Q_{\underline{\psi}}f\|_{T^{p,2}(H;X)} =
  \|f\|_{H^{p}_{A,\underline{\psi}}(H;X)}.
\end{equation*}
Taking $\phi=1$, this gives (i). Taking $\phi\in \Psi(S_\theta)$, we obtain
(ii). 
\end{proof}

The following, by now quite simple result has some useful consequences:

\begin{proposition}
If \((0,\widetilde\psi)\) has sufficient decay, then the bounded mapping \(S_{\widetilde\psi}: T^{p,2}(H;X)\to H^p_A(H;X)\)
is surjective.
\end{proposition}

\begin{proof}
By Remark~\ref{rem:decay}, we find a
\(\psi\in \Psi(S_{\theta})\)
such that~\eqref{eq:SQ} is satisfied and \((\psi,\widetilde\psi)\) has sufficient decay.
Now let \(f\in H_A^p(H;X) = H^p_{A,\psi}(H;X)\) be arbitrary
and let $\limn f_n = f$ in $H_{A,\psi}^p(H;X)$ with  
\(f_n\in \overline{\ran(A)}\otimes X\). 
The functions $g_n:=Q_{\psi} f_n$ belong to $T^{p,2}(H;X)$
and $\n g_n - g_m \n_{T^{p,2}(H;X)} = \n f_n - f_m\n_{H_{A,\psi}^p(H;X)}$
for all $m,n$. It follows that the sequence $(f_n)$ is Cauchy in 
$T^{p,2}(H;X)$ and therefore converges to some $f\in T^{p,2}(H;X)$.
From \(f_n =  S_{\widetilde\psi} g_n\) and the continuity of \(S_{\widetilde\psi}\)
it follows that \(f = S_{\widetilde\psi} g\).
\end{proof}

\begin{corollary}
Let \((0,\widetilde\psi)\) have sufficient decay. An equivalent description of the Hardy space is
\begin{equation*}
  H^p_A(H;X)=\widetilde{H}^p_{A,\widetilde\psi}(H;X):=\{S_{\widetilde\psi}F:F\in T^{p,2}(H;X)\},
\end{equation*}
and an equivalent norm is given by
\begin{equation*}
  \|f\|_{\widetilde{H}^p_{A,\widetilde\psi}(H;X)}:=\inf\{\|F\|_{T^{p,2}(H;X)}:f=S_{\widetilde\psi}F\}.
\end{equation*}
\end{corollary}

As a further consequence we deduce an interpolation result for Hardy spaces from the following general principle (see Theorem 1.2.4 in \cite{triebel}):
Let $X_0,X_1$ and $Y_0,Y_1$ be two interpolation couples such that
 there exist operators \(S \in \mathcal{L}(Y_{i},X_{i})\) and
\(Q \in \mathcal{L}(X_{i},Y_{i})\) with \(SQx =x\) for all \(x \in X_{i}\) and \(i=0,1\). 
Then $[X_0,X_1]_\theta = S[Y_0,Y_1]_\theta$.
Here we take \((\psi,\widetilde{\psi})\) as in the Calder\'on reproducing formula with sufficient decay, \(S=S_{\widetilde\psi}\) and \(Q=Q_{{\psi}}\).

\begin{corollary} Let $H$ be a Hilbert space and $X$ be a UMD space.
For all \(1<p_0<p_1<\infty\) and \(0<\theta<1\) we have
\begin{equation*}
  [H_A^{p_0}(H;X), H_A^{p_1}(H;X)]_\theta = H_A^{p_\theta}(H;X),
  \qquad \frac{1}{p_\theta}=\frac{1-\theta}{p_0}+\frac{\theta}{p_1}.
\end{equation*}
\end{corollary}

\section{Hardy spaces associated with differential operators}
\label{sect:diffop}
The construction described in Section \ref{sect:hardy} is particularly relevant when dealing with differential operators \(A=D_B\) in \(L^{2}(\C \oplus \C^{n})\), where
\begin{equation*}
 D_{B} = \Big( \begin{array}{cc}0 & -{\rm div} B \\ 
               \nabla & 0\end{array}\Big)
\end{equation*}
with \(B\) a multiplication operator on \(L^2(\C^n)\) given by an $(n\times n)$-matrix with 
\(L^{\infty}\) entries.
Such operators have been considered in connection with 
the celebrated square root problem of Kato, which was originally solved in~\cite{AHLMT}. A new proof based on first order methods was devised in \cite{AKM}, where it was shown that $D_B$ bisectorial on $L^2(\C\oplus \C^n)$ and satisfies off-diagonal estimates of any order.

In \cite{HMP}, the $H^\infty$-functional calculus of $D_B\otimes I_{X}$ in $L^p(X\oplus X^n)$  
is described in terms of $R$-boundedness of the resolvents.
Although these resolvent conditions, and hence the functional calculus, may fail
on \(L^{p}(X\oplus X^n)\) in general, it follows from Section \ref{sect:hardy} that these operators do
have an $H^\infty$-functional calculus on 
\(H^{p}_{D_{B}}(\C\oplus\C^{n} ; X)\), which in particular implies Kato type estimates in this space.

To express these estimates, 
observe first that \(\overline{\ran(D_B)}=\overline{\ran(\mathrm{div} B)}\oplus\overline{\ran(\nabla)}\).
Let us hence write a function \(f\in\overline{\ran(D_B)}\otimes X\) as \((f_0,f_1)\), where
\begin{equation*}
  f_0\in\overline{\ran(\mathrm{div} B)}\otimes X\subseteq L^2(\C)\otimes X,\qquad
  f_1\in\overline{\ran(\nabla)}\otimes X\subseteq L^2(\C^n)\otimes X
\end{equation*}
denote the \(X\)-valued and \(X^{n}\)-valued parts of \(f\), respectively.
Defining
\begin{equation*}
  H^{p}_{D_{B}}(\C\oplus\C^{n} ; X) := H^{p}_{D_{B},\psi}(\C\oplus\C^{n} ; X)
\end{equation*} 
by means of the (even!) function 
\(\psi(z) = (\sqrt{z^{2}})^{N}e^{-\sqrt{z^{2}}}\) with $N$ large enough, 
we note that $\psi(t D_B) = \phi(t^2D_B^2)$, where \(\phi(z) =
\sqrt{z}^N e^{-{\sqrt{z}}}\) and the operator
\begin{equation*}
  D_B^2 =  \Big( \begin{array}{cc}
   -{\rm div}B\nabla & 0 \\ 0 & -\nabla{\rm div}B\end{array}\Big),
\end{equation*}
and hence \(\phi(t^2 D_B^2)\), is diagonal with respect to the splitting \(f=(f_0,f_1)\). In particular this shows that
\begin{equation*}
  \|(f_0,f_1)\|_{H^{p}_{D_{B}}(\C\oplus\C^{n} ; X)}
 \eqsim
  \|(f_0,0)\|_{H^{p}_{D_{B}}(\C\oplus\C^{n} ; X)}
  +\|(0,f_1)\|_{H^{p}_{D_{B}}(\C\oplus\C^{n} ; X)}.
\end{equation*}
Hence also the full space \(H^p_{D_B}(\C\oplus\C^n;X)\) (constructed as the completion of \(\overline{\ran(D_B)}\otimes X\) with respect to the above-given norm) has the natural direct sum splitting into ``\(X\)-valued'' and ``\(X^n\)-valued'' components. Let us denote these components by \(H^{p}_{D_{B}}(\C; X)\) and \(H^{p}_{D_{B}}(\C^n; X)\), so that
\begin{equation*}\begin{split}
  \n f_0\n_{H^{p} _{D_{B}}(\C; X)} &:= \n (f_0,0) \n_{ H^{p}_{D_{B}}(\C\oplus\C^{n} ; X)},\\
  \n f_1\n_{H^{p} _{D_{B}}(\C^{n} ; X)} &:= \n (0,f_1)\n_{H^{p}_{D_{B}}(\C\oplus\C^{n} ; X)}.
\end{split}\end{equation*}
Then we are ready to state:

\begin{theorem}
Let \(X\) be a UMD space, \(1<p<\infty\), and \(D_{B}\) be as above. Then 
\begin{equation*}
 \|\sqrt{-{\rm div} B\nabla}u\|_{H^{p}_{D_{B}}(\C ; X)} 
  \eqsim \|\nabla u\|_{H^{p}_{D_{B}}(\C^{n} ; X)}
\end{equation*}
for all \(u\in\dom(\nabla)\otimes X\subset L^2(\C)\otimes X\).
\end{theorem}

\begin{proof}
We know from~\cite{AKM} that \((I+zD_B)^{-1}\) satisfies off-diagonal estimates of arbitrary order and that \(D_B\) has an 
\(H^{\infty}(S_{\theta})\)-calculus 
on \(L^2(\C\oplus\C^n)\).

Consider the function \(\phi(z)=z/\sqrt{z^2}\in H^{\infty}(S_{\theta})\). By the boundedness of the 
\(H^{\infty}\)-calculus 
and the identity \(1/\phi(z)=\phi(z)\), 
\begin{equation}\label{eq:phiDB}
  \|\phi(D_B)f\|_{H^p_{D_B}(\C\oplus\C^n;X)}\eqsim\|f\|_{H^p_{D_B}(\C\oplus\C^n;X)},\qquad
  f\in \overline{\ran(D_B)}\otimes X.
\end{equation}
Observing that
\begin{equation*}
  \phi(D_B)=\begin{pmatrix}
   0 & -\mathrm{div}B(-\nabla B\mathrm{div})^{-1/2} \\
   \nabla(-\mathrm{div}B\nabla)^{-1/2} & 0
  \end{pmatrix}
\end{equation*}
and writing \eqref{eq:phiDB} for \(f=(f_0,0)\) gives
\begin{equation}\label{eq:fnought}
  \|\nabla(-\mathrm{div}B\nabla)^{-1/2}f_0\|_{H^p_{D_B}(\C^n;X)}
  \eqsim\|f_0\|_{H^p_{D_B}(\C;X)},\qquad f_0\in \overline{\ran(\mathrm{div}B)}\otimes X.
\end{equation}

Let then \(u\in\dom(\nabla)\otimes X\). By the solution of Kato's problem we have \(\dom(\nabla)=\dom(\sqrt{-\mathrm{div}B\nabla})\). Substituting 
\begin{equation*}
  f_0=\sqrt{-\mathrm{div}B\nabla}u\in\ran(\sqrt{-\mathrm{div}B\nabla})\otimes X
  \subseteq\overline{\ran(-\mathrm{div}B\nabla)}\otimes X
  \subseteq\overline{\ran(\mathrm{div}B)}\otimes X
\end{equation*}
in \eqref{eq:fnought}, we obtain the assertion. We used above the inclusion \(\ran(A^{1/2})\subseteq\overline{\ran(A)}\), which is true for all sectorial operators (see~\cite{Haase}, Corollary~3.1.11).
\end{proof}


Let \(D\) be the unperturbed operator \(D_{I}\). Observe that \(D^{2}(f,0)= (\Delta f,0)\) and then, whenever \(\psi\) is even, \(\psi(tD)(f,0)=(\psi(t\sqrt{\Delta})f,0)\).
The space \(H^{p}_{D}(\C,X)\) is then the classical Hardy space.

\begin{theorem}
\label{thm:hisl}
Let \(X\) be UMD. Then \(H^{p}_{D}(\C,X) = L^{p}(X)\) for all \(1<p<\infty\).
\end{theorem}
\begin{proof}
Let us denote by \(N\) the smallest integer greater than \(\frac{n}{2}\) and, 
for functions \(f \in C_c\), define 
$$Sf(y,t) = \int  _{\R^{n}} k_{t}(y,z)f(z)\ud z,$$ where
$$k_{t}(y,z) = t^{N} \frac{\partial^{N}}{\partial
t^{N}}\Big(t^{-n}p\big(\frac{(y-z)}{t}\big)\Big),$$ and
\(p(w)=1/{(1+w^{2})^{\frac{n+1}{2}}}.\)
For a fixed \(t>0\), \(f \mapsto Sf(\cdot,t)\) is thus a Fourier multiplier with symbol \(m_t(\xi)=(t|\xi|)^{N}e^{-t|\xi|}\). This implies assumptions (1) and (4) in 
Theorem \ref{thm:sing}. Assumptions (2) and (3), with \(\alpha = \beta = 1\),
follow from direct computations of the $N$-th derivative of 
\(t \mapsto t^{-n}p(\frac{|x|}{t})\) and the mean value theorem.
Now, for \(f \in L^{p}(X)\), 
letting 
$$Pf(y,t) :=
\psi(tD)(f,0)(y) = ((t\sqrt{\Delta})^{N}e^{-t\sqrt{\Delta}}f(y),0)$$ 
and applying Theorem \ref{thm:sing}, we thus obtain that
$$\|f\|_{H^{p}_{D}(\C,X)} \lesssim \|f\|_{L^{p}(X)}$$ for all \(f \in L^{p}(X)\).
Now let $f \in L^{p}(X)$ and $g \in L^{p'}(X^{*})$, and denote by $\langle f,g \rangle$ their duality product.
By Calder\'on's reproducing formula there exists $\widetilde{\psi}$ (with arbitrary decay) such that
\begin{equation*}
   \langle f,g \rangle = 
  \int_{0}^{\infty}\langle \psi(t\Delta)f,\widetilde{\psi}(t\Delta)^{*}g \rangle \frac{\ud t}{t}.
\end{equation*}
Therefore
\begin{equation*}\begin{split}
  |\langle f,g \rangle|
  \lesssim \|f\|_{H^{p}_{D}(\C,X)}\|g\|_{H^{p'}_{D}(\C;X^{*})}
  \lesssim \|f\|_{H^{p}_{D}(\C,X)}\|g\|_{L^{p'}(X)},
\end{split}\end{equation*}
and hence \(\|f\|_{L^p(X)}\lesssim \|f\|_{H^{p}_{D}(\C,X)}\).
\end{proof}

\bibliography{tent}{}
\bibliographystyle{plain}

\end{document}